\documentclass {siamltex}
\usepackage{graphicx,xcolor}
\usepackage{epsfig}
\usepackage{latexsym}
\usepackage{array}
\usepackage{amssymb}
\usepackage{amsmath}
\usepackage{amsbsy}
\usepackage{stmaryrd}

\newcommand{\jump}[1]{\llbracket#1\rrbracket}

\newtheorem{remark}{\it Remark\/}
\usepackage{eepic}
\title{Stabilized finite element methods for nonsymmetric,
  noncoercive and ill-posed problem. Part II: hyperbolic equations}

\author{Erik Burman\thanks{Department of Mathematics, 
University College London, London, 
UK--WC1E  6BT, 
United Kingdom; ({\tt e.burman@ucl.ac.uk})}
}
 
\begin{document}

\maketitle

\begin{abstract}
In this paper we consider stabilized finite element methods for hyperbolic transport equations without
coercivity. Abstract conditions for the convergence of the methods are
introduced and these conditions are shown to hold for three different
stabilized methods: the Galerkin least squares method, the continuous
interior penalty method and the discontinuous Galerkin method.
We consider both the standard stabilization methods and the
optimisation based method introduced in \cite{part1}. 
The main idea of the latter is to write the stabilized method in an optimisation
framework and select the discrete function for which a certain cost
functional, in our case stabilization term, is minimised.
Some numerical examples illustrate the theoretical investigations.
\end{abstract}


\newcommand{\cut}{c}
\def\IR{\mathbb R}
\def\Ext{\mbox{\textsf{E}}}

\section{Introduction}
Several finite element methods have been proposed for the computation
of hyperbolic problems, such as the SUPG method \cite{BH82,JNP84}, the
discontinuous Galerkin method \cite{RH73, LR74, JP86} and several
different weakly consistent, symmetric stabilization methods for
continuous approximation spaces, \cite{Gue99, Cod00, BH04, BB04}. In
most of these cases however the analysis relies on the satisfaction of a
coercivity condition. Indeed if a scalar hyperbolic transport equation
\begin{equation}\label{model_problem}
\beta \cdot \nabla u + \sigma u = f
\end{equation}
is considered, with data given on the inflow boundary, it is typically
assumed that there exists $\sigma_0 \in \mathbb{R}^+$ such that 
\begin{equation}\label{coercivity}
\sigma_0 \leq \inf_{x \in \Omega} \left(\sigma - \frac12 \nabla \cdot \beta\right).
\end{equation}
In for instance \cite{JNP84,JP86,AM09} the degenerate case
$\sigma_0=0$ is allowed using special exponentially weighted test
functions, which we will also exploit in this paper.

In practice this condition is quite restrictive and rules out many
important flow regimes such as exothermic reactions, compressible
flow fields or data assimilation problems with data given on the
outflow boundary. Our objective in the present paper is to propose an
analysis of stabilized finite element methods in the noncoercive case. Indeed
similarly as in the elliptic case \cite{Schatz74} the discrete
solutions of standard stabilized finite element methods are shown
to exist and have optimal convergence under a condition on the mesh
size. 
Unlike the elliptic case there appears to be no equivalent
result, even suboptimal, for the standard Galerkin method. This part
uses tools similar to those of \cite{JNP84,JP86,AM09}. Then we
show how the method
introduced in \cite{part1} can be applied to hyperbolic problems 
beyond the coercive regime of condition \eqref{coercivity}. The advantage of
this latter method is
that the mesh conditions under which the analysis holds are much less
restrictive and boundary conditions may be imposed on the outflow
boundary just as easily as on the inflow boundary, without modifying
the parameters of the method.
For a full motivation of
the method and analysis in the elliptic case we refer the reader to
\cite{part1}.

We will
consider the problem \eqref{model_problem} with smooth coefficients,
$\beta \in [W^{2,\infty}(\Omega)]^d$ and $\sigma \in 
W^{1,\infty}(\Omega)$. Boundary data will be given either on the inflow or the
outflow correponding to solving either the standard transport problem or
a model data assimilation problem. For such smooth physical parameters both cases can easily
be solved using the method of characteristics, provided that for each
$x \in \Omega$ there exists a streamline leading, in finite time, to the boundary where
data is imposed and $|\beta(x)|\ne 0$ for all $x \in \Omega$. In the
following we always assume that $\beta$ satisfies these
assumptions, unless otherwise stated, and that the stationary problem admits a unique, sufficiently
smooth solution. 

Problems on conservation form $\nabla \cdot (\beta u)$ are cast on the form \eqref{model_problem} by using the
product rule and including the low order term with coefficient $\nabla
\cdot \beta$ in $\sigma$. The present paper has the following structure. In section
\ref{sec:abstract} we propose an abstract analysis under certain
assumptions on the discrete bilinear form. Then in section
\ref{sec:stabilization} we give a detailed description of how three
different stabilization methods, the Galerkin least squares method (GLS),
the continuous interior penalty (CIP) method and the discontinuous Galerkin
method (DG) satisfy the assumptions of the abstract theory for the
case of the advection--reaction equation.  In all cases we prove that
the classical quasi optimal estimate
for stabilized methods holds
\[
\|u - u_h\|_{L^2(\Omega)} + \|h^{\frac12} \beta \cdot \nabla (u -
u_h)\|_{L^2(\Omega)} \leq C h^{k+\frac12} |u|_{H^{k+1}(\Omega)}.
\]
We also show how to include a model
problem for data assimilation in the analysis. Finally in section
\ref{numerics} we illustrate  the theory with some numerical examples.

\section{Abstract formulation}\label{sec:abstract}
Let $\Omega$ be a polygonal/polyhedral subset of $\mathbb{R}^d$. The
boundary of $\Omega$ will be denoted by $\partial \Omega$ and its
outward pointing normal by $n$.
We let $V,W$ denote two Hilbert spaces with norms $\|\cdot\|_V$ and $\|\cdot\|_W$.
The abstract weak formulation of the continuous problem takes
the form: find $u \in V$ such that 
\begin{equation}\label{forward}
a(u,v) = (f,v), \quad \forall v \in W
\end{equation}
with formal adjoint: find $z \in W$ such that
\begin{equation}\label{adjoint}
a(w,z) = (g, w) , \quad \forall w \in V.
\end{equation}
The bilinear form $a(\cdot,\cdot):V\times W
\rightarrow \mathbb{R}$ and the data $f$ are assumed to satisfy the assumptions of
Babuska's theorem \cite{Bab70} so that the problems \eqref{forward} and
\eqref{adjoint} are well posed. (See \cite{EG06} for an analysis of
\eqref{model_problem} in the coercive regime.) We denote the forward problem on strong form
$\mathcal{L} u = f$ and the adjoint problem on strong form
$\mathcal{L}^* z = g$. 
\begin{remark}
The analysis below never uses the full power of Babuska's theorem. We
only need to assume that \eqref{forward} admits a unique solution for
the given data and that certain discrete stability conditions are
satisfied by $a(\cdot,\cdot)$ as specified below. For the problems
considered herein the solution of \eqref{adjoint} will always be
$z=0$.
\end{remark}

\subsection{Finite element discretisation}
Let $\{\mathcal{T}_h\}_h$ denote a family of quasi uniform, shape
regular triangulations $\mathcal{T}_h:=\{K\}$, indexed by the maximum
triangle radius $h:=\max_{K \in \mathcal{T}_h} h_K$. The set of faces of the
triangulation will be denoted by $\mathcal{F}$ and
$\mathcal{F}_{int}$ denotes the subset of interior faces.  Let $X_h^k$ denote the finite element space of piecewise 
polynomial functions on
$\mathcal{T}_h$, 
$$
X_h^k := \{v_h \in L^2(\Omega): v_h\vert_{K} \in \mathbb{P}_k(K),\quad \forall K
\in \mathcal{T}_h\}.
$$
Here $\mathbb{P}_k(K)$ denotes the space of polynomials of degree
less than or equal to $k$ on a triangle $K$.  The $L^2$-scalar product over some {measurable} $X \subset \mathbb{R}^d$ is denoted $(\cdot,\cdot)_X$ and the associated
norm $\|\cdot\|_X$, the subscript is dropped whenever $X = \Omega$. We
will also use $\left<\cdot,\cdot\right>_Y$ to denote the $L^2$-scalar
product over $Y \subset \mathbb{R}^{d-1}$. For the element wise
$L^2$-scalar product and norm over $\Omega$ we will use the notation
$(\cdot,\cdot)_h:= \sum_{K \in \mathcal{T}_h} (\cdot,\cdot)_K$,
$\|\cdot\|_h:= (\cdot,\cdot)_h^{\frac12}$. 
In the estimates of the paper capital constants are
generic, whereas lower case constants are specific to the estimate. 
Sometimes capital constants will be given subscripts to point
to the main dependencies on parameters. We will also use $a \sim b$ to
stress an important dependence in $a$ on some parameter $b$, i.
e. $a = C b$, with $C$ assumed to be moderate.

We let $\pi_L$ denote the standard $L^2$-projection onto $X_h^k$ and
$i_h:C^0(\bar \Omega) \mapsto X_h^k$ the standard Lagrange interpolant. Recall that
for any function $u \in (V \cup W)\cap H^{k+1}(\Omega)$ there holds
\begin{equation}\label{approx}
\|u - i_h u\| + h \|\nabla(u - i_h u)\|+ h^2 \|D^2(u - i_h u)\|_h  \leq c_i h^{k+1} |u|_{H^{k+1}(\Omega)},
\end{equation}
where $D^2$ denotes the Hessian matrix and the matrix norm used is the
Frobenius norm. A similar result holds for $\pi_L$. If $\pi_L$ projects onto $X_h^k \cap
C^0(\bar \Omega)$ the same result holds under the assumption of local
quasi regularity of the mesh. The following discrete commutator
property follows by straightforward modifications of the result in
\cite{Berto99} and holds for $i_h$, the element-wise $L^2$-projection
onto $X_h^k$ and, under our assumptions
on the mesh, for the $L^2$-projection onto continuous finite element
functions. Here $\varphi \in W^{2,\infty}(\Omega)$, $0\leq n \leq 2$,
\begin{equation}\label{discrete_commutator}
\sum_{K\in \mathcal{T}_h} |\varphi u_h - i_h(\varphi u_h) |^2_{H^n(K)} \leq c^2_{dc,n,\varphi}
h^{-2n+2} \|u_h\|^2_{L^2(\Omega)}.
\end{equation}
We also note that the following inverse inequalities hold, $\exists
c_T, c_I \in \mathbb{R}^+$ such that
\begin{equation}\label{inverse_eq}
\begin{array}{l}
\|u\|_{\partial K} \leq c_T (h^{-\frac12} \|u\|_{K} + h^{\frac12}
\|\nabla u\|_K), \quad \forall u \in H^1(K)\\[3mm]
h_K^{-\frac12} \|u_h\|_{\partial K} + h_K \|\nabla u_h\|_K  \leq c_I
\|u_h\|_{K}, \quad \forall u_h \in \mathbb{P}_k(K).
\end{array}
\end{equation}
Let $V_h$ and $W_h$ denote two finite element spaces such that
$\mbox{dim} ~ V_h=  \mbox{dim} ~ W_h$ (in practice $V_h = W_h$  herein). Now we introduce a discrete
bilinear form $a_h(\cdot,\cdot):V_h \times W_h \mapsto \mathbb{R}$
associated to $a(\cdot,\cdot)$ 
and a stabilization operator $s_p(\cdot,\cdot):V_h \times W_h \mapsto \mathbb{R}$. The standard stabilized finite element
formulation for the problem \eqref{forward} takes the form, find $u_h
\in V_h$ such that
\begin{equation}\label{standardFEM}
a_h(u_h,v_h) + s_p(u_h,v_h) = (f,v_h) + s_p(u,v_h) \quad \forall v_h
\in W_h.
\end{equation}
Observe that since $s_p(u,v_h)$ appears in the right hand side,
we can only use stabilization operators such that this quantity is known.
As we shall see below, the noncoercivity of the form
$a_h(\cdot,\cdot)$ leads to problem dependent mesh conditions for the
well-posedness of \eqref{standardFEM}. To alleviate the conditions on the mesh we propose the following finite element method for the approximation
of \eqref{forward}, find $(u_h,z_h) \in V_h \times W_h$ such that
\begin{equation}\label{FEM}
\begin{array}{rcl}
a_h(u_h,w_h) + s_a(z_h,w_h) &=&(f,w_h) \\[3mm]
a_h(v_h,z_h) - s_p(u_h,v_h) &=& -s_p(u,v_h),
\end{array}
\end{equation}
for all $(v_h,w_h) \in V_h \times W_h$. Here $s_a(\cdot,\cdot)$ is a
stabilization term related to the adjoint equation that will be
discussed below. {Observe that we here solve
simultaneously \eqref{forward} and \eqref{adjoint}, with $g=0$ in the
latter equation. 
We will consider either continuous approximation spaces,
$V_h:=X^k_h\cap H^1(\Omega)$ or discontinuous approximation $V_h := X^k_h$.
The bilinear form
$a_h(\cdot,\cdot)$ is a discrete realisation of $a(\cdot,\cdot)$,
typically modified to account for the effect of nonconformity,
since in general $V_h \not \subset V$ and $W_h \not \subset W$. Weakly
imposed boundary conditions may be set in the form $a_h(\cdot,\cdot)$,
but below we have chosen to
impose them using $s_p(\cdot,\cdot)$ and $s_a(\cdot,\cdot)$ to obtain
a more unified analysis. In 
\eqref{FEM} stabilization can also be added in $a_h(\cdot,\cdot)$.
Our numerical experiments did not show any advantages of the addition
and this approach will not be pursued herein.

The bilinear forms
$s_a(\cdot,\cdot)$, $s_p(\cdot,\cdot)$ in \eqref{FEM} are  symmetric, positive semi-definite, stabilization
operators, defined on $[V_h \cup W_h]^2$. For
simplicity we will always assume that $u$ is
sufficiently regular so that strong consistency holds,
i.e. $s_p(u,v_h)$  is well defined. Note also that for the method to
make sense $s_p(u,v_h)$ must be known, either to be zero, or depending
only on known data. This will be the case below. The modifications of the 
analysis to the case of weakly consistent stabilization are
straightforward and not considered herein. The semi-norm on $V_h\cup W_h$ associated to the stabilization is defined by
\[
|x_h|_{S_y} := s_y(x_h,x_h)^{\frac12}, \quad y = a,p.
\]
We will
assume that the following strong consistency property holds. If $u$
is the solution of \eqref{forward} then
\begin{equation}\label{consist1}
a_h(u,\varphi) =(\mathcal{L }u,\varphi) = (f,\varphi) \mbox{ for all }\varphi \in W_h.
\end{equation}
Then $u$ solution of \eqref{forward} solves \eqref{standardFEM}, 
and $u$ solution of \eqref{forward} and $z \equiv 0$ solve the
system \eqref{FEM}.

We also assume that there are interpolation operators $\pi_V:V 
\rightarrow V_h $ and $\pi_W : W \rightarrow W_h$, satisfying
\eqref{approx}. We introduce the (semi-)norm $\|\cdot\|_+$ and assume
that the following approximation estimates are satisfied 
\begin{equation}\label{approx1}
\|v - \pi_V v\|_V + \|v - \pi_V v\|_+ + |v- \pi_V v|_{S_p}  \leq c_{a\gamma} h^r
|v|_{H^{k+1}(\Omega)}, \quad \forall v \in V \cap H^{k+1}(\Omega),
\end{equation} 
where $r>0$, depends {on the approximation properties of the finite
element space and the definition of the norms in the left hand
side}. From
the standard error estimates for stabilized methods we expect
$r=k+\frac12$ for smooth exact solutions. The
constant $c_{a\gamma}$ depends on the form $a(\cdot,\cdot)$ and 
stabilization parameter(s) of the method included in
$s_p(\cdot,\cdot)$ and $s_a(\cdot,\cdot)$, here denoted $\gamma$. 
\subsection{Abstract assumptions on the formulation
  \eqref{standardFEM}}
The assumptions made below consititutes sufficient conditions for the
method \eqref{standardFEM} to converge. Here we assume that
$\|\cdot\|_V \equiv \|\cdot \|_W$. As usual the conditions are
consistency, stability and continuity of the forms.
First consistency, Galerkin orthogonality for \eqref{standardFEM} is a consequence of the consistency \eqref{consist1} 
\begin{equation}\label{galortho1}
a_h(u - u_h,w_h) + s_p(u-u_h,w_h) = 0, \quad \forall w_h \in W_h.
\end{equation}
We assume that there exists $c_s, c_\eta \in \mathbb{R}^+$ such that
for all $h>0$ and $u_h \in V_h$ there exists $v_a \in W_h$ satisfying
\begin{equation}\label{strong_bound_va}
c_s (\|u_h\|^2_V + |u_h|^2_{S_p}) \leq a_h(u_h,v_a(u_h)) + s_p(u_h,v_a(u_h)) + \epsilon(h) (\|u_h\|_V^2 + |u_h|^2_{S_p}),
\end{equation}
where $\epsilon(h)$ is a continuous function such that $\epsilon(0)=0$,
and 
\begin{equation}\label{stab_v_bound3}
\|v_{a}(u_h)\|_V + |v_a(u_h)|_{S_p}
\leq c_\eta (\|u_h\|_V + |u_h|_{S_p}).
\end{equation}
These assumptions ensure that the stabilized formulation satisfies a
discrete inf-sup condition for $\epsilon(h)$ small enough. We also assume the following continuity.
\begin{equation}\label{cont2}
a_h(v - \pi_V v, x_h) \leq \|v - \pi_V v\|_+ c_{a}
(|x_h|_{S_p}+\|x_h\|_V),\, \forall v \in V, \, x_h \in W_h.
\end{equation}
\subsection{Abstract assumptions on the formulation \eqref{FEM}}
Observe that the following partial coercivity is obtained
by taking $v_h = u_h$ and $w_h=z_h$ in \eqref{FEM},
\begin{equation}\label{partial_coerc}
|z_h|_{S_a}^2 + |u_h|_{S_p}^2= (f, z_h) {+ s_p(u,u_h)}.
\end{equation} 
The following Galerkin orthogonality holds for \eqref{FEM} by \eqref{consist1},
\begin{equation}\label{galortho2}
\begin{array}{c}
a_h(u - u_h,w_h) = s_a(z_h,w_h)\quad\forall w_h \in W_h \\
{a_h(v_h, z_h) =
s_p(u_h - u,v_h)} \quad \forall v_h \in V_h.
\end{array}
\end{equation}
Let $\tilde \epsilon(h)$ and $\breve \epsilon(h)$ denote continuous,
monotonically increasing functions such that
$\tilde \epsilon(0)=0$ and $0 \leq \breve \epsilon(h)$.
We assume that the following discrete stability holds for all
$u_h\in V_h$, $z_h \in W_h$. For some
$\tilde c_s, \tilde c_\eta \in \mathbb{R}^+$, for all $u_h \in V_h$, there exists $v_{a}(u_h) \in W_h$ such that
\begin{equation}\label{forward_stability}
\tilde c_s \|u_h\|^2_V \leq a_h(u_h,v_{a}(u_h)) + \tilde
\epsilon(h)\|u_h\|^2_V + \tilde c_\eta|u_h|^2_{S_p}
\end{equation}
and similarly, for all $z_h \in W_h$ there exists $v_{a*}(z_h) \in V_h$ such that
\begin{equation}\label{adjoint_stability}
\tilde c_s \|z_h\|^2_W \leq a_h(v_{a*}(z_h),z_h) +  \tilde
\epsilon(h)\|z_h\|^2_W + \tilde c_\eta |z_h|^2_{S_a}.
\end{equation}
Moreover assume that the functions $v_a$ and $v_{a*}$ satisfy the bounds
 \begin{equation}\label{stab_v_bound2}
\|v_{a}(u_h)\|_W  \leq \tilde c_\eta \|u_h\|_V,  \quad |v_{a}(u_h)|_{S_a}
\leq \breve \epsilon(h) \|u_h\|_V +  \tilde c_\eta |u_h|_{S_p},
\end{equation}
\begin{equation}\label{stab_v_bound1}
\|v_{a*}(z_h)\|_V  \leq \tilde c_\eta \|z_h\|_W,  \quad |v_{a*}(z_h)|_{S_p}
\leq\breve \epsilon(h) \|z_h\|_W +  \tilde c_{\eta} |z_h|_{S_a}.
\end{equation}
Since we are interested in problems that are ill-conditioned, we here
assume $\tilde c_s<\tilde c_\eta$ without loss of generality.
We finally assume that the following continuity relation 
holds
\begin{equation}\label{cont1}
a_h(v - \pi_V v, x_h) \leq \|v - \pi_V v\|_+ c_{a}
(|x_h|_{S_a}+  \|x_h\|_W),\, \forall v \in V, \, x_h \in W_h.
\end{equation}
\subsection{Convergence analysis for the abstract methods}
We will first prove a convergence result for the standard stabilized
finite element method \eqref{standardFEM}. Then we will consider
\eqref{FEM}. 
\begin{proposition}\label{standard_conv}
Assume that the solution of \eqref{forward} is smooth and that the forms of
   \eqref{standardFEM} and the operators $\pi_V$, $\pi_W$ 
    are such that \eqref{galortho1}--\eqref{cont2} are satisfied.
Also assume that $\epsilon(h)$ satisfies the bound, 
\begin{equation}\label{h_cond_strong}
\epsilon(h) \leq \frac{c_s}{2}.
\end{equation}
Then \eqref{standardFEM} admits a unique solution $u_h$ for which there holds
\[
\|u - u_h\|_V + |u - u_h|_{S_p} \leq c_{as\gamma}
h^r
|u|_{H^{k+1}(\Omega)},
\]
where $c_{as\gamma}\sim (c_a+1) \frac{c_\eta}{c_s}$.
\end{proposition}
\begin{proof}
Since the spaces $W_h$ and $V_h$ have the same dimension, the matrix
is square and it is sufficient to prove uniqueness.
Assume $(f,v_h)+s_p(u,v_h) = 0$ for all $v_h \in W_h$. Under the
condition \eqref{h_cond_strong} there holds
\[
\frac12 c_s (\|u_h\|^2_V + |u_h|^2_{S_p}) \leq a_h(u_h,v_a(u_h)) +
s_p(u_h,v_a(u_h)) = 0
\]
hence $u_h = 0$ and existence and uniqueness follows.
Let $\xi_h := \pi_V u - u_h$. By the stability assumption 
\eqref{strong_bound_va} we have
\begin{equation*}
c_s( \|\xi_h\|_V^2 +  |\xi_h|^2_{S_p})\leq
a_h(\xi_h,v_a(\xi_h))+s_p(\xi_h,v_a(\xi_h))+ \epsilon(h) ( \|\xi_h\|_V^2 +  |\xi_h|^2_{S_p}).
\end{equation*}
It follows that under the condition \eqref{h_cond_strong} there holds
\[
\frac12 c_s( \|\xi_h\|_V^2 + |\xi_h|^2_{S_p}) \leq
a_h(\xi_h,v_a(\xi_h))+s_p(\xi_h,v_a(\xi_h))
\]
and by Galerkin orthogonality \eqref{galortho1}, the continuity \eqref{cont2} and the
stability \eqref{stab_v_bound3}
\begin{multline*}
\frac12 c_s( \|\xi_h\|_V^2 +  |\xi_h|^2_{S_p}) \leq
a_h(\pi_V u - u,v_a(\xi_h))+s_p(\pi_V u - u,v_a(\xi_h))\\
\leq  c_a \|\pi_V u - u\|_+(|v_a(\xi_h)|_{S_p} + \|v_a(\xi_h)\|_V) +
|\pi_V u - u|_{S_p} |v_a(\xi_h)|_{S_p}\\
\leq (c_a+1)( \|\pi_V u - u\|_+ + |\pi_V u - u|_{S_p})  c_\eta( \|\xi_h\|_V +  |\xi_h|_{S_p}).
\end{multline*}
We conclude by noting that $\|u - u_h\|_V \leq \|u-\pi_V u\|_V +
\|\xi_h\|_V$ and  applying the approximation \eqref{approx1}.
\end{proof}\\
We now turn to the analysis of \eqref{FEM}. In this case the analysis is based on a combination of coercivity of the
stabilization operators \eqref{partial_coerc} and an inf-sup argument
using \eqref{forward_stability} and \eqref{adjoint_stability}. This
allows us to exploit the strong stability property \eqref{partial_coerc} enjoyed by
the stabilization terms and thereby improve the robustness of the method.
\begin{theorem}\label{stab_conv}
Assume that
   the solution of \eqref{forward} is smooth, that the forms of
   \eqref{FEM} and the operators $\pi_V$, $\pi_W$ 
    are such that \eqref{approx1}, \eqref{galortho2}--\eqref{cont1} are satisfied and that
\begin{equation}\label{h_cond}
\tilde \epsilon(h) \leq \frac{\tilde c_s}{2}.
\end{equation}
Then \eqref{FEM} admits a unique solution $u_h,z_h$ for which there holds
\[
\|u - u_h\|_V + \|z_h\|_W + |u - u_h|_{S_p}  + |z_h|_{S_a}\leq \tilde c_{as\gamma}
h^r
|u|_{H^{k+1}(\Omega)}.
\]
The constant in the above estimate is given by
$$
\tilde c_{as\gamma} \sim (c_a + 1)\frac{\tilde c_{\eta}}{\tilde c_s} \left(1 +
  \frac{\breve \epsilon(h)^2}{\tilde c_{\eta} \tilde c_s} \right).
$$
Similarly, if $s_p(u,w_h)=0$, there holds
\[
 |u_h|_{S_p} + |z_h|_{S_a}  \leq \tilde c_{as\gamma}  h^r
|u|_{H^{k+1}(\Omega)}.
\]
\end{theorem}
\begin{proof}
For the first inequality, let $\xi_h = \pi_V u - u_h$.  As in the
previous case it is enough to prove the claim for $\xi_h$. By the
definition \eqref{FEM} there holds
\[
|\xi_h|_{S_p}^2+|z_h|_{S_a}^2   = s_p(\xi_h,\xi_h)+s_a(z_h,z_h) = a_h(\xi_h,z_h) +
s_a(z_h,z_h) -  a_h(\xi_h,z_h) + s_p(\xi_h,\xi_h).
\]
By the stabilities \eqref{forward_stability}--\eqref{stab_v_bound2} there exists $v_a(\xi_h)$ and
$v_{a*}(z_h)$ such that
\begin{multline*}
\tilde c_s (\|\xi_h\|^2_V +\|z_h\|_W^2) \leq a_h(\xi_h,v_a(\xi_h)) +
s_a(v_a(\xi_h),z_h) \\+ a_h(v_{a*}(z_h),z_h) -
s_p(\xi_h, v_{a*}(z_h))
+ 
\tilde \epsilon(h)\|\xi_h\|^2_V + \tilde c_\eta
|\xi_h|^2_{S_p}\\
+|z_h|_{S_a} (\breve \epsilon(h) \|\xi_h\|_V + \tilde c_{\eta} |\xi_h|_{S_p}) 
+  
\tilde \epsilon(h)\|z_h\|^2_W +  \tilde c_\eta
|z_h|_{S_a}^2\\
+|\xi_h|_{S_p}(\breve \epsilon(h)\|z_h\|_W +   \tilde c_{\eta} |z_h|_{S_a}) 
.
\end{multline*}
It follows that for all $\mu_V,\mu_S>0$ we may write
\begin{multline*}
\tilde c_s \mu_V (\|\xi_h\|^2_V +\|z_h\|_W^2)
+\mu_S(|\xi_h|_{S_p}^2+|z_h|_{S_a}^2) \leq a_h(\xi_h, \mu_S z_h + \mu_V
v_a(\xi_h)) \\ +
s_a(\mu_S z_h + \mu_V v_a(\xi_h),z_h) - a_h(\mu_S \xi_h - \mu_V
v_{a*}(z_h),z_h) + s_p(\xi_h ,\mu_S \xi_h - \mu_V v_{a*}(z_h))
\\
+ \mu_V \tilde \epsilon(h)(\|\xi_h\|^2_V + \|z_h\|^2_W)+  \mu_V \tilde c_\eta ( |\xi_h|^2_{S_p}+|z_h|^2_{S_a})\\
+\mu_V |z_h|_{S_a} (\breve \epsilon(h) \|\xi_h\|_V + \tilde c_{\eta} |\xi_h|_{S_p}) 
+\mu_V |\xi_h|_{S_p}(\breve \epsilon(h)\|z_h\|_W +   \tilde c_{\eta} |z_h|_{S_a})  .
\end{multline*}
By arithmetic-geometric inequalities in the right hand side 
\begin{multline*}
\mu_V \tilde \epsilon(h)(\|\xi_h\|^2_V + \|z_h\|^2_W)+  \mu_V \tilde c_\eta ( |\xi_h|^2_{S_p}+|z_h|^2_{S_a})\\
+\mu_V |z_h|_{S_a} (\breve \epsilon(h) \|\xi_h\|_V + \tilde c_{\eta} |\xi_h|_{S_p}) 
+\mu_V |\xi_h|_{S_p}(\breve \epsilon(h)\|z_h\|_W +   \tilde c_{\eta} |z_h|_{S_a}) \\
\leq \mu_V \left(\tilde \epsilon(h) + \frac14 \tilde c_s\right) (\|\xi_h\|_V^2+\|z_h\|_W^2)
+ \mu_V \left(2 \tilde c_\eta  + \frac{\breve \epsilon(h)^2}{ \tilde c_s}\right) \  (|\xi_h|_{S_p}^2 +  |z_h|^2_{S_a}).
\end{multline*}
Therefore under
the condition \eqref{h_cond} there holds
\begin{multline*}
\frac14 \tilde c_s\mu_V (\|\xi_h\|^2_V +\|z_h\|_W^2)
+ \left(\mu_S -   \mu_V \left(2 \tilde c_\eta  + \frac{\breve \epsilon(h)^2}{ \tilde c_s}\right) \right)(
|\xi_h|_{S_p}^2+|z_h|_{S_a}^2) \\
\leq a_h(\xi_h,\mu_S z_h + \mu_V
v_a(\xi_h)) +
s_a(\mu_S  z_h + \mu_V v_a(\xi_h),z_h) \\ -  a_h(\mu_S\xi_h - \mu_V
v_{a*}(z_h),z_h) + s_p(\xi_h,\mu_S\xi_h - \mu v_{a*}(z_h)).
\end{multline*}
Then, by choosing $\mu_V = \tfrac{4}{\tilde c_s}$, $\mu_S = 
\tfrac{9 \tilde c_\eta}{\tilde c_s}  + \tfrac{4 \breve \epsilon(h)^2}{ \tilde c_s^2}$
and applying the Galerkin orthogonality of equation 
\eqref{galortho2}, we have, since by assumption $\tilde c_s < \tilde c_\eta$
\begin{multline*}
\|\xi_h\|^2_V +\|z_h\|_W^2 +
|\xi_h|_{S_p}^2+|z_h|_{S_a}^2\\ \leq a_h(\pi_V u -
u,\mu_S z_h + \mu_V
v_a(\xi_h)) + s_p(\pi_V u-u, \mu_S \xi_h - \mu_V v_{a*}(z_h)).
\end{multline*}
We proceed by applying the continuity \eqref{cont1} in the first term
of the right hand side and the Cauchy-Schwarz inequality in the
stabilization term,
\begin{multline*}
\|\xi_h\|^2_V
+\|z_h\|_W^2+|\xi_h|_{S_p}^2+|z_h|_{S_a}^2\\ \leq \|u - \pi_V u\|_+
c_a(|\mu_S z_h + \mu_V
v_a(\xi_h)|_{S_a} + \|\mu_S z_h + \mu_V
v_a(\xi_h)\|_W)\\+ |u-\pi_V u|_{S_p}|\mu_S \xi_h - \mu_V v_{a*}(z_h)|_{S_p}.
\end{multline*}
Using a triangle inequality followed by the  the stability of $v_a$
\eqref{stab_v_bound2} and $v_{a*}$ \eqref{stab_v_bound1} and the bound
$\mu_V (\tilde c_\eta + \breve \epsilon(h)) < \mu_S$, that holds under
the assumption $\tilde c_s < \tilde c_{\eta}$,
we may conclude that
\begin{multline*}
\|\xi_h\|^2_V
+\|z_h\|_W^2+|\xi_h|_{S_p}^2+|z_h|_{S_a}^2 \leq
 (\|u - \pi_V u\|_+ +  |u-\pi_V u|_{S_p})\\ \times
(c_a + 1) \mu_S (   \|\xi_h\|_V
+\|z_h\|_W+|\xi_h|_{S_p} + |z_h|_{S_a} ).
\end{multline*}
We conclude from this expression and \eqref{approx1} that the first
claim holds.
The second result is an immediate consequence of $s_p(u,w_h)=0$ and
the symmetry of $s_p(\cdot,\cdot)$. Uniqueness of the discrete
solution follows by taking $f=0$ in \eqref{forward} and observing that
since then $u=\pi_V u=0$ we have $u_h=z_h=0$ by which uniqueness
follows using the same a priori estimates.
\end{proof}
\section{Stabilization methods}\label{sec:stabilization}
We let $\mathcal{L}$ denote the first order hyperbolic operator on non-conservation form,
\begin{equation}\label{hyperbolic_oper}
\mathcal{L} u := \beta \cdot \nabla u + \sigma u.
\end{equation}
Here $\beta \in [W^{2,\infty}(\Omega)]^d$ is a non-solenoidal velocity
vectorfield and $\sigma \in W^{1,\infty}(\Omega)$. 
We also assume that boundary conditions are set on the inflow boundary
$\partial \Omega^-$,
$$
u\vert_{\partial \Omega^-} = g_{in}, \quad \partial \Omega^\pm := \{x \in \partial \Omega : \pm \beta(x) \cdot n > 0 \}.
$$ 
The adjoint operator
takes the form
\begin{equation}\label{hyperbolic_adjoint}
\mathcal{L}^* u := -\nabla \cdot(\beta u) + \sigma u.
\end{equation}
We have assumed below that the reaction is moderately stiff
so that the relevant time scale of the flow is given by $h
|\beta|^{-1}$. In particular we will not track the influence of the size
of $\sigma$ in the error bounds below, assuming
$h^{\frac12}(\|\sigma\|_{L^\infty(\Omega)} + \|\nabla \cdot \beta\|_{L^\infty(\Omega)})$ moderate.
We will consider three
different stabilized finite element methods below and show that they
all satisfy the assumptions of the abstract theory. The bilinear form
$a_h(\cdot,\cdot)$ of \eqref{standardFEM} and \eqref{FEM}
is defined as
\begin{equation}\label{discrete_bilin}
a_h(u_h,v_h) := (\mathcal{L} u_h,v_h)_h - \frac12 \sum_{K \in \mathcal{T}_h}
\int_{\partial K \setminus \partial \Omega} \beta \cdot n_{\partial K}[u_h] \{v_h\} ~\mbox{d}s
\end{equation}
where $\{v_h\}$ denotes the average of $v_h$ from the two element
faces, 
\[
\{u_h\}(x)\vert_{\partial K} := \frac12 \lim_{\varepsilon \rightarrow 0^+} (u_h(x-\varepsilon
n_{\partial K}) + u_h(x+\varepsilon
n_{\partial K})),
\]
the jump of $u_h$ is defined as
\[
[u_h](x)\vert_{\partial K} := \lim_{\varepsilon \rightarrow 0^+} (u_h(x-\varepsilon
n_{\partial K}) - u_h(x+\varepsilon
n_{\partial K})).
\]
As usual the jump terms on $u_h$ may be omitted when a
    continuous function is considered in the formulation.
First we will prove a general stability result on $a_h(u_h,v_h)$. 
\begin{lemma}\label{basic_stability}
For the bilinear form \eqref{discrete_bilin} there holds $\forall \eta
\in W^{1,\infty}(\Omega)$, $\forall u_h,z_h \in X_h^k$,
\[
a_h(u_h,e^{\pm \eta} u_h) = \frac12 \int_{\partial \Omega} (\beta \cdot
n) u_h^2 e^{\pm\eta} ~\mbox{d}s + \int_{\Omega} u_h^2 \left(\mp \frac12 \beta
\cdot \nabla \eta - \frac12 \nabla \cdot \beta + \sigma\right) e^{\pm\eta}~\mbox{d}x,
\]
\[
a_h( e^{\pm \eta}z_h, z_h) = \frac12 \int_{\partial \Omega} (\beta \cdot
n) z_h^2 e^{\pm \eta} ~\mbox{d}s + \int_{\Omega} z_h^2 \left(\pm \frac12 \beta
\cdot \nabla \eta - \frac12 \nabla \cdot \beta + \sigma\right) e^{\pm \eta}~\mbox{d}x.
\]
\end{lemma}
\begin{proof}
Consider the first inequality with the negative sign in the exponent.
By definition we have
\begin{multline}\label{a_defstart}
a_h(u_h,e^{-\eta} u_h)  = (\beta \cdot \nabla  u_h+\sigma u_h,
e^{-\eta} u_h)_h - \frac12  \sum_{K \in \mathcal{T}_h}
\int_{\partial K \setminus \partial \Omega} \beta \cdot n_{\partial K}
[u_h] \{e^{-\eta} u_h\} ~\mbox{d}s 
\end{multline} 
and note that an integration by parts in the advective term yields
\begin{multline*}
(\beta \cdot \nabla  u_h, e^{-\eta} u_h)_h - \frac12 \sum_{K \in \mathcal{T}_h}
\int_{\partial K \setminus \partial \Omega} \beta \cdot n_{\partial K}
[u_h] \{e^{-\eta} u_h\} ~\mbox{d}s = (u_h,e^{-\eta} (\beta \cdot
\nabla \eta - \nabla \cdot \beta) u_h)\\
 - (u_h,
e^{-\eta} \beta \cdot \nabla u_h)_h + \frac12 \sum_{K \in \mathcal{T}_h}
\int_{\partial K \setminus \partial \Omega} \beta \cdot n_{\partial K}
[u_h] \{e^{-\eta} u_h\} ~\mbox{d}s \\
+ \int_{\partial \Omega} (\beta \cdot
 n) u_h^2 e^{-\eta} ~\mbox{d}s.
\end{multline*}
This equality implies the following well-known relation,
\begin{multline}\label{eq:dg_cons}
(\beta \cdot \nabla  u_h, e^{-\eta} u_h)_h - \frac12 \sum_{K \in \mathcal{T}_h}
\int_{\partial K \setminus \partial \Omega} \beta \cdot n_{\partial K}
[u_h] \{e^{-\eta} u_h\} ~\mbox{d}s \\
 = \frac12 \left( (u_h,e^{-\eta} (\beta \cdot
\nabla \eta - \nabla \cdot \beta) u_h)  + \int_{\partial \Omega} (\beta \cdot
 n) u_h^2 e^{-\eta} ~\mbox{d}s\right).
\end{multline}
The first stability result is obtained by applying this equality in
\eqref{a_defstart}.
The inequality for the adjoint case is proven similarly by observing that after
an integration by parts in the bilinear form
\begin{multline}\label{z_defstart}
a_h(e^{-\eta} z_h, z_h)  = -(e^{-\eta} z_h, \beta \cdot \nabla  z_h +
(\nabla \cdot \beta - \sigma) z_h) \\
+ \frac12\sum_{K \in \mathcal{T}_h}
\int_{\partial K \setminus \partial \Omega} \beta \cdot n_{\partial K}
[z_h] \{e^{-\eta} z_h\} ~\mbox{d}s 
+ \int_{\partial \Omega} (\beta \cdot
 n) z_h^2 e^{-\eta} ~\mbox{d}s
\end{multline} 
and then applying \eqref{eq:dg_cons}. The case where the power is
positive follows similarly, observing that the change of sign only has
an effect in the inner derivative $\beta \cdot \nabla \eta$.
\end{proof}\\
The importance of this Lemma is a consequence of the existence of a
particular function $\eta$ that is given in the following result. 
\begin{lemma}
Under the assumptions on $\beta$ there exists $\eta_0
\in W^{2,\infty}(\Omega)$ such that $\beta \cdot \nabla \eta_0 \ge 1$ in
$\Omega$.
\end{lemma}
For the proof of this result see \cite[Appendix A]{AM09}.

It follows that the second term of the right hand sides in the equations of Lemma
\ref{basic_stability} are {non-negative} for  
\begin{equation}\label{eq:stab_func}
\eta := (1+ \|2 \sigma-\nabla \cdot  \beta
\|_{L^\infty(\Omega)}) \, \eta_0.
\end{equation}
Below we
always assume that $\eta$ is of this form.
In general $e^{-\eta} u_h \not \in V_h$ and hence Lemma
\ref{basic_stability} is insufficient to prove
\eqref{forward_stability} and \eqref{adjoint_stability}. The trick is
to chose $v_a$ to be some suitable approximation of $e^{-\eta} u_h$ in
$V_h$, $\pi e^{-\eta} u_h$, and control the approximation error using the
stabilization. Since we are often required to estimate this error we
introduce the notation $\delta(e^{-\eta} u_h) := e^{-\eta} u_h - \pi e^{-\eta} u_h$. Similarly $v_{a*}$ is chosen as an approximation of $-
e^{-\eta} z_h$. 

The stabilization terms may now be chosen as
one of the following, where the first two assumes $H^1$-conforming
approximation and the last discontinuous approximation. In all three
cases we have $W_h \equiv V_h$. Below $\gamma_X \in \mathbb{R}^+$, $X
= GLS, \, CIP,\, DG$,
denotes a stabilization parameter associated to the method $X$ and
$\gamma_{bc} \in \mathbb{R}^+$ a stabilization parameter associated to
the weakly imposed boundary condition.
\begin{itemize}
\item The Galerkin least squares method.\\
In this case continuous finite element spaces are used, $V_h = W_h := X_h^k \cap
H^1(\Omega)$ and the stabilization operators take the form
\begin{equation}\label{GLS_prim}
s_{p,GLS}(u_h,w_h) := (\gamma_{GLS} |\beta|^{-1} h \mathcal{L} u_h, \mathcal{L} w_h) ,
\end{equation}
\begin{equation}\label{GLS_adjoint}
s_{a,GLS}(z_h,v_h) := (\gamma_{GLS} |\beta|^{-1} h \mathcal{L^*} z_h, \mathcal{L^*} v_h).
\end{equation}
Note that $s_{p,GLS}(u,w_h) =
(f,\gamma_{GLS} |\beta|^{-1} h \mathcal{L} w_h)$ showing that
$s_p(u,\cdot)$ can indeed be expressed using data. 
\item Continuous interior penalty stabilization.\\
Here as well continuous finite element spaces are used, $V_h = W_h  := X_h^k \cap
H^1(\Omega)$ and the stabilization is given by
\begin{equation}\label{jumpstab1}
s_{CIP}(u_h,w_h) := \sum_{F \in\mathcal{F}_{int}} \int_F h^2_F
\gamma_{CIP} \|\beta_h \cdot n_F\|_{L^\infty(F)}
\jump{\nabla u_h} \cdot \jump{\nabla w_h} ~\mbox{d}x
\end{equation}
 for both the primal and the adjoint equations, where 
$\jump{\nabla u_h}\vert_F$ denotes the
jump of the gradient over the
face $F$. 
\item The discontinuous Galerkin method.\\
In this case we do not impose any continuity constraints in the
finite element space
$V_h := X_h^k $. The method is stabilized by penalising the jump of the solution
over element faces for both the primal and the adjoint equations.
\begin{equation}\label{DGstab}
s_{DG}(u_h,w_h) := \sum_{F \in\mathcal{F}_{int}} \int_F 
\gamma_{DG} |\beta \cdot n_F|
[u_h] [w_h] ~\mbox{d}x.
\end{equation}
where 
$[u_h]\vert_F$ denotes the
jump of the solution over the
face $F$. The choice $\gamma_{DG} = \tfrac12$ is known to lead to the
classical upwind formulation for the method \eqref{standardFEM}.
\end{itemize}
To
account for boundary conditions the above stabilizations are modified as
follows
\begin{equation}\label{stab_ops_bound}
\begin{aligned}
s_p(u_h,w_h)& := s_{p,X}(u_h,w_h) + s_{bc,-}(u_h,w_h),\\
s_a(z_h,v_h) &:= s_{a,X}(z_h,v_h) + s_{bc,+}(z_h,v_h)+s_{bc,-}(z_h,v_h),
\end{aligned}
\end{equation}
with $X =
GLS,\,CIP,\, DG$ and
$
s_{bc,\pm} := \int_{\partial \Omega}\gamma_{bc}  |(\beta \cdot n)_{\pm}|
u_h v_h ~\mbox{d}s.
$
Note that the value of $z_h$ is penalized on the whole boundary. This
is necessary to obtain robustness if no boundary conditions are set
in $a_h(\cdot,\cdot)$ and allows for the simple choice of test functions used in the
analysis below. 
It should be noted that for problems where the adjoint solution satisfies $z=0$ the
stabilization in the bulk or on the boundary can be changed to any form satisfying the
assumptions \eqref{adjoint_stability}-\eqref{cont1}. The
consistency requirements are much weaker, since the exact solution is
trivial. The variant where $z_h$ is penalized only on the outflow
boundary can also be shown to be stable using the arguments below
provided that weak boundary
conditions are included also in $a_h(\cdot,\cdot)$. In this case
different weight functions must be used for $u_h$ and $z_h$. The present choice
was motivated mainly by the use of a single exponential weight in all
estimates and that it makes integration
of data assimilation problems straightforward, by changing the
boundary contribution in $s_p(\cdot,\cdot)$.

Below we will consider the methods \eqref{standardFEM} and \eqref{FEM}
one by one, in each case showing that the
assumptions \eqref{galortho1}-\eqref{cont2} are satisfied
for method \eqref{standardFEM} as well as
\eqref{galortho2}--\eqref{cont1} for method \eqref{FEM}. Clearly some arguments are
very similar between the different methods and full details are given
only for the GLS-method. The conclusion is that all three schemes
satisfy the assumptions necessary for the abstract analysis to
hold. The dependence of the $\epsilon(h)$, $\tilde \epsilon(h)$,
$\breve \epsilon(h)$ and
$c_\eta$ and $\tilde c_\eta$ on the physical parameters
and on $h$ is specified in each case in the proofs. The natural norm for the
analysis is
\[
\|x\|_W = \|x\|_V := \|x \| + \|h^{\frac12} \beta
\cdot \nabla x \|_h + \| |\beta
\cdot n|^{\frac12} x  \|_{\partial \Omega},
\]
but to keep down the technical detail we will first prove the results
in the reduced norm,
\begin{equation}\label{red_norm}
\|x\|_W = \|x\|_V := \|x  \| + \| |\beta
\cdot n|^{\frac12} x \|_{\partial \Omega}
\end{equation}
and then show how the control of the streamline derivative can be
recovered separately. We also define the continuity norm for all three
methods as
\begin{equation}\label{dualnorm}
\|v\|_+:= \|(|\beta|^{\frac12} h^{-\frac12} + |\sigma_\beta|) v\|+
\||\beta \cdot n|^{\frac12} v\|_{\mathcal{F}},
\end{equation}
where $\sigma_\beta = - \nabla \cdot \beta + \sigma$.
It is straightforward to show that in all cases the approximation
estimate \eqref{approx1} holds with $r = k+{\frac12}$ for any
interpolant in $X_h^k$
with optimal approximation properties.
The error estimate that results from the abstract analysis for the
transport equation may be
written in all cases, for both \eqref{standardFEM} and \eqref{FEM},
\[
\|u - u_h\|_V + \|h^{\frac12} \beta \cdot \nabla (u - u_h)\| +
|u-u_h|_{S_p} \leq C h^{k+\frac12} |u|_{H^{k+1}(\Omega)}.
\]
However the condition \eqref{h_cond_strong} leads to a stronger constraint on the mesh for
the  formulation \eqref{standardFEM} than \eqref{h_cond}. We first prove a Lemma, similar to the superapproximation
result of \cite{JP86}, useful in all three cases. 
\begin{lemma}\label{interp_lem}
Let $\pi$ be an interpolation operator that satisfies \eqref{approx} and \eqref{discrete_commutator}
then there holds
\[
 \|e^{-\eta} u_h - \pi e^{-\eta} u_h\|_{V} + \|e^{-\eta} u_h - \pi
 e^{-\eta} u_h\|_+ + |e^{-\eta} u_h - \pi e^{-\eta} u_h|_{S_x} \leq
 \Pi(h) \|u_h\|_V,
\]
where $x=a,p$ and $ \Pi(h)  = C_{\gamma\beta\sigma}
c_{dc,e^{-\eta}} \,h^{\frac12}$. Here $c_{dc,e^{-\eta}}$ refers to the maximum
constant of \eqref{discrete_commutator} for $n=0,1,2$. The result
holds for all the three methods presented above.
\end{lemma}
\begin{proof}
First observe that by inequality \eqref{discrete_commutator} we have
\begin{equation}\label{bas_eu_est}
 h^{-\frac12} \|\delta(e^{-\eta} u_h)\| + h^{\frac12}
 \|\nabla\delta(e^{-\eta} u_h)\| \leq C \max_{n \in
   \{0,1\}}c_{dc,n,e^{-\eta}} h^{\frac12} \|u_h\|,
\end{equation}
recalling that $\delta(e^{-\eta} u_h) := e^{-\eta} u_h - \pi e^{-\eta} u_h$.
Similarly using \eqref{inverse_eq} followed by
\eqref{discrete_commutator} gives
\begin{multline*}
\sum_{K \in \mathcal{T}_h}\|\delta(e^{-\eta} u_h)\|^2_{\partial K}
\leq \sum_{K \in \mathcal{T}_h} c_T^2(
h^{-\frac12} \|\delta(e^{-\eta} u_h)\|^2_{K}  + h^{\frac12}
\|\nabla\delta(e^{-\eta} u_h)\|^2_{ K} )\\ \leq
C^2 \max_{n \in
   \{0,1\}}c^2_{dc,n,e^{-\eta}} h \|u_h\|^2.
\end{multline*}
Using these results in the definitions \eqref{red_norm} and \eqref{dualnorm} we obtain
\begin{multline*}
\|\delta(e^{-\eta} u_h)\|_+ + \|\delta(e^{-\eta} u_h)\|_V 
\leq
C (\|\beta\|_{L^\infty}^{\frac12} + h^{\frac12} \|\sigma_\beta\|
_{L^\infty}^{\frac12} + h^{\frac12} ) \|h^{-\frac12} \delta(e^{-\eta} u_h)\|\\+\| |\beta \cdot n|^{\frac12}\delta(e^{-\eta} u_h)\|_{\mathcal{F}} \leq C_{\beta\sigma} \max_{n \in
   \{0,1\}} c_{dc,n,e^{-\eta}} h^{\frac12} \|u_h\|.
\end{multline*}
For the stabilization norm we first consider the boundary term and
 the three methods separately. 
For the boundary terms we observe that
\[
s_{bc,\pm}(\delta(e^{-\eta} u_h),\delta(e^{-\eta} u_h))^{\frac12} \leq
\gamma_{bc}^{\frac12}\|\delta(e^{-\eta} u_h)\|_V \leq C_{\gamma\beta\sigma} \max_{n \in
   \{0,1\}} c_{dc,n,e^{-\eta}} h^{\frac12} \|u_h\|.
\]
Then note that for the GLS method
\begin{multline*}
s_{p,GLS}(\delta(e^{-\eta} u_h),\delta(e^{-\eta} u_h))^{\frac12} \leq \gamma_{GLS}^{\frac12} h^{\frac12} (
\|\beta\|^{\frac12}_{L^\infty}\|\nabla \delta(e^{-\eta} u_h)\|_h + \|\sigma\|_{L^\infty} \|\delta(e^{-\eta} u_h)\| )\\
\leq C_{\gamma\beta\sigma} \max_{n \in
   \{0,1\}} c_{dc,n,e^{-\eta}} h^{\frac12} \|u_h\|
\end{multline*}
and similarly
$
s_{a,GLS}(\delta(e^{-\eta} u_h),\delta(e^{-\eta} u_h))^{\frac12} \leq  C_{\gamma\beta\sigma_{\beta}} \max_{n \in
   \{0,1\}} c_{dc,n,e^{-\eta}} h^{\frac12} \|u_h\|.
$

For the CIP-method we use element-wise trace inequalities followed by
\eqref{discrete_commutator}, with $n=1$ and $n=2$,
\begin{multline*}
s_{CIP}(\delta(e^{-\eta} u_h),\delta(e^{-\eta} u_h))^{\frac12} \\
\leq \gamma_{CIP}^{\frac12} c_T h^{\frac12}
\|\beta\|_{L^\infty} (\sum_{K \in \mathcal{T}_h} (\|\nabla
\delta(e^{-\eta} u_h)\|_K^2 + h^2 \|D^2 \delta(e^{-\eta} u_h)\|_K^2
))^{\frac12} \\
\leq \gamma_{CIP}^{\frac12} c_T h^{\frac12}
\|\beta\|_{L^\infty} (c_{dc,1,e^{-\eta}} + c_{dc,2,e^{-\eta}}) \|u_h\|
\end{multline*}

Finally for the DG-method, we simply observe that 
\[
s_{DG}(\delta(e^{-\eta} u_h),\delta(e^{-\eta} u_h))^{\frac12} \leq C \gamma_{DG}\|\delta(e^{-\eta} u_h)\|_+.
\]
\end{proof}
\subsection{Galerkin-least-squares stabilization}\label{subsec:GLS}
We assume that $$V_h = X_h^k \cap H^1(\Omega),\, W_h=V_h$$ 
Let $\pi_V,\, \pi_W$ be defined by the Lagrange
interpolator $i_h$. It follows by the construction of the
stabilization operator and \eqref{consist1}  that \eqref{galortho1}
and \eqref{galortho2} hold (recalling that $z\equiv 0$.) It is also
straightforward to show that \eqref{approx1} holds with $r=k+\frac12$.
We collect the proof of the remaining assumptions of Proposition \ref{standard_conv} and Theorem
\ref{stab_conv} in two propositions.
\begin{proposition}(Satisfaction of assumptions for \eqref{standardFEM} with GLS)\label{assumpGLS}
Let the bilinear forms of \eqref{standardFEM} be defined by
\eqref{discrete_bilin} and \eqref{GLS_prim} with $\gamma_{bc}\ge 1$. Then \eqref{strong_bound_va}--
\eqref{cont2} are
satisfied, with $\epsilon(h) = C_{\gamma\beta\sigma\eta} h^{\frac12}$.
\end{proposition}
\begin{proof}
To show \eqref{strong_bound_va} we take $v_a := \pi_V (e^{-\eta} u_h)$ with
$\eta$ defined by \eqref{eq:stab_func} and use first inequality of Lemma \ref{basic_stability} to obtain
\begin{multline}\label{bas_stab_GLS}
 a_h(u_h,\pi_V (e^{-\eta} u_h)) =  a_h(u_h, e^{-\eta} u_h)  - a_h(u_h,\delta (e^{-\eta} u_h))
\\ \ge -\gamma_{GLS}^{-\frac12} |u_h|_{S_p}
\|\delta (e^{-\eta} u_h)\|_+
\\+ \frac12 \int_{\partial \Omega} (\beta \cdot n) u_h^2
e^{-\eta} ~\mbox{d}s+ 
\frac12 \|u_h e^{-\frac{\eta}{2}} \|^2.
\end{multline}
Using Lemma \ref{interp_lem} we have
\begin{multline}\label{a_stab_part}
\frac12 \|u_h e^{-{\frac{\eta}{2}}} \|^2  + \frac12 \int_{\partial \Omega} (\beta \cdot n)_+ u_h^2
e^{-\eta} ~\mbox{d}s \leq  a_h(u_h,\pi_V (e^{-\eta} u_h)) \\
 - \frac12 \int_{\partial \Omega} (\beta \cdot n)_- u_h^2 e^{-\eta}
 +\frac12 \gamma_{GLS}^{-\frac12} \Pi(h) (|u_h|^2_{S_p} + \|u_h\|^2_V).
\end{multline}
We need a similar bound for the stabilization operator using the function
$v_a(u_h)$. This is straightforward observing that
\begin{multline*}
s_{p}(u_h,v_a(u_h)) = (\mathcal{L} u_h
,\gamma_{GLS} |\beta|^{-1} h \mathcal{L} (u_h e^{-\eta}))  + s_{bc,-}(u_h, e^{-\eta} u_h) \\-
(\mathcal{L} u_h , \gamma_{GLS} |\beta|^{-1} h (\mathcal{L} \delta (e^{-\eta} u_h))
- s_{bc,-}(u_h, \delta (e^{-\eta} u_h)))
\\
\ge  \| (\gamma_{GLS} h |\beta|^{-1})^{\frac12} \mathcal{L} u_h
e^{-\frac{\eta}{2}}\|^2+ \gamma_{bc} \|| (\beta \cdot n)_-|^{\frac12} u_h
e^{-\frac{\eta}{2}}\|^2_{\partial \Omega} \\ - |u_h|_{S_p}\left(| \delta (e^{-\eta} u_h)|_{S_p}+ \| (\gamma_{GLS} h |\beta|^{-1})^{\frac12}
(\mathcal{L} e^{-\frac{\eta}{2}} )u_h\|\right).
\end{multline*}
Combining this result with \eqref{a_stab_part}, using
\eqref{interp_lem} it follows that for $\gamma_{bc}$ large enough
\begin{multline}\label{a_s_tot}
\frac12 \inf_{x \in \Omega} e^{-\eta}  (\|u_h\|_V^2  + |u_h|^2_{S_p}) \leq  a_h(u_h,\pi_V
(e^{-\eta} u_h)) 
+ s_{p,GLS}(u_h,\pi_V (e^{-\eta} u_h)) \\
+(C_{\gamma}\Pi(h) + (\gamma_{GLS} h |\beta|^{-1})^{\frac12}
\sup_{x \in \Omega} |\mathcal{L} e^{-\frac{\eta}{2}} |) (|u_h|^2_{S_p} + \|u_h\|^2_V).
\end{multline}
We conclude that \eqref{strong_bound_va} holds with $c_s = \tfrac12
\inf_{x \in \Omega} e^{-\eta}$
and 
\begin{equation*}
\epsilon(h)= (C_{\gamma}\Pi(h) + (\gamma_{GLS} h |\beta|^{-1})^{\frac12}
\sup_{x \in \Omega} |\mathcal{L} e^{-\frac{\eta}{2}} |)  \sim
 C_{\gamma\beta\sigma\eta} h^{\frac12}.
\end{equation*}
Considering now \eqref{stab_v_bound3} we have
\begin{equation}\label{1st_arg_stab_v}
\|v_{a}(u_h)\|_V \leq \|e^{-\eta} u_h\|_V + \|\delta(e^{-\eta} u_h) \|_V \leq (\sup_{x \in \Omega} e^{-\eta} + \Pi(h) ) \|u_h\|_V
\end{equation}
and for the stabilization part,
\begin{multline}\label{1st_arg_stab_s}
|v_{a}(u_h)|_{S_p} \leq   \sup_{x \in \Omega}  |\mathcal{L} e^{-\frac{\eta}{2}} |  h^{\frac12} C_{\gamma}
\|u_h\|_V +\sup_{x \in \Omega} e^{-\eta}|u_h|_{S_p} + |\delta(e^{-\eta} u_h)|_{S_p} \\
\leq (\sup_{x \in \Omega} e^{-\eta}  h^{\frac12}
C_{\gamma\beta\sigma\eta}+\Pi(h))\|u_h\|_V  + \sup_{x \in \Omega} e^{-\eta} |u_h|_{S_p}.
\end{multline}
It follow that \eqref{stab_v_bound3} holds for any $$c_\eta \ge \max(\sup_{x \in \Omega}  |\mathcal{L} e^{-\frac{\eta}{2}} |  h^{\frac12}
C_{\gamma\beta\sigma}, \sup_{x \in \Omega} e^{-\eta} +\Pi(h))
\sim C_{\gamma\beta\sigma\eta} h^{\frac12} + \sup_{x \in \Omega} e^{-\eta}.$$
For the continuity \eqref{cont2} we first use an
integration by parts and Cauchy-Schwarz inequality to obtain
\begin{multline}\label{bas_continuity}
a_h(v - \pi_V v, x_h) =   (u - \pi_V u, \mathcal{L}^* x_h)+
\int_{\partial \Omega} (\beta \cdot n)( v - \pi_V v) x_h ~\mbox{d}s\\
\leq
\|u - \pi_V u\|_+ (\|(|\beta|^{-1} h)^{\frac12}
\mathcal{L}^* x_h\| + \|x_h\|_V).
\end{multline}
To conclude we need to express the norm over the adjoint operator in
the right hand side by the stabilization of the primal
operator. Observe that for all $x_h \in V_h$ there holds
\begin{equation}\label{LstaraboundbySp}
\||(\beta|^{-1} h)^{\frac12}
\mathcal{L}^* x_h\|\leq |x_h|_{S_p} + C_{\gamma\beta}( h^{\frac12}
(2 \|\sigma\|_{L^\infty(\Omega)}  + \|\nabla \cdot \beta
\|_{L^\infty(\Omega)})) \|x_h\|_V.
\end{equation}
Collecting the results of \eqref{bas_continuity} and
\eqref{LstaraboundbySp} we see that $c_a \ge 1+ C_{\gamma\beta\sigma} h^{\frac12}$.
\end{proof}
\begin{proposition}(Satisfaction of the assumptions for \eqref{FEM}
  with GLS)\label{assumpGLS_new}
 Let the bilinear forms of \eqref{FEM} be defined by
\eqref{discrete_bilin}, \eqref{GLS_prim} and
\eqref{GLS_adjoint}. Then the
inequalities \eqref{galortho2}--\eqref{cont1} hold with $\tilde \epsilon(h) = 0
$.
\end{proposition}
\begin{proof}
Starting from the inequality \eqref{bas_stab_GLS} with $v_a(u_h) := \pi_W(e^{-\eta} u_h)$ we immediately get,
\begin{multline*}
\frac12 \inf_{x\in \Omega} e^{-\eta} \|u_h\|^2_V \leq a_h(u_h,\pi_W (e^{-\eta} u_h)) +
\gamma_{GLS}^{-\frac12} \Pi(h)
|u_h|_{S_p} \|u_h\|_V \\ + \sup_{x\in \Omega} e^{-\eta} \gamma_{bc}^{-1}|u_h|_{S_p}^2
\end{multline*}
from which we deduce, using $(\inf_{x\in \Omega} e^{-\eta} )^{-1} = \sup_{x\in \Omega} e^{\eta} $,
\begin{equation*}
\frac14 \inf_{x\in \Omega} e^{-\eta} \|u_h\|^2_V \leq a_h(u_h,\pi_W (e^{-\eta} u_h)) +(\sup_{x\in \Omega} e^{\eta} \gamma_{GLS}^{-1} \Pi(h)^2+\sup_{x\in \Omega} e^{-\eta}\gamma_{bc}^{-1} )|u_h|_{S_p}^2
\end{equation*}
which is the required inequality with $\tilde
\epsilon(h)=0$, $\tilde c_s = \tfrac14  \inf_{x\in \Omega} e^{-\eta} $and $$\tilde c_\eta \ge \sup_{x\in \Omega} e^{\eta}
\gamma_{GLS}^{-1} \Pi(h)^2+\sup_{x\in \Omega} e^{-\eta}\gamma_{bc}^{-1}.$$

 In a similar fashion we may show that
\eqref{adjoint_stability}
holds, also with the weight $e^{-\eta}$, and corresponding test
function $v_{a*}(z_h) = -\pi_V (e^{-\eta} z_h)$. First observe that in
this case using Lemma \ref{basic_stability} (second equation), 
\begin{multline*}
\frac12 \inf_{x\in \Omega} e^{-\eta} \|z_h\|^2 - \frac12 \int_{\partial \Omega}
(\beta \cdot n) z_h^2 e^{-\eta} ~\mbox{d}s \leq -a_h(e^{-\eta}
z_h,z_h) \\
= a_h(-\pi_V (e^{-\eta} z_h),z_h) -  a_h(\delta (e^{-\eta} z_h),z_h).
\end{multline*}
For the second term in the right hand side we have after integration
by parts and application of Lemma \ref{interp_lem}
\begin{multline*}
a_h(\delta (e^{-\eta} z_h),z_h) = \int_{\partial \Omega} (\beta \cdot
n) \delta (e^{-\eta} z_h) z_h ~\mbox{d}s + (\delta (e^{-\eta} z_h),
\mathcal{L^*} z_h) \\
\leq C \|\delta (e^{-\eta} z_h)\|_+ |z_h|_{S_a} \leq C \Pi(h)
\|z_h\|_V |z_h|_{S_a}.
\end{multline*}
Here we used that the boundary penalty on $z_h$ is active on the
whole boundary. We may then conclude as before that
\[
\frac14 \inf_{x\in \Omega} e^{-\eta} \|z_h\|^2_W \leq a_h(-\pi_V (e^{-\eta} z_h),z_h) 
+(\sup_{x\in \Omega} e^{\eta} C_\gamma \Pi(h)^2+\sup_{x\in \Omega} e^{-\eta} \gamma_{bc}^{-1}) |z_h|_{S_a}^2
\]
with similar constants as before. 

The inequalities of \eqref{stab_v_bound2} and \eqref{stab_v_bound1} follow by similar arguments as
\eqref{1st_arg_stab_v} and \eqref{1st_arg_stab_s}. The only differences occur in the right inequalities.
\begin{multline*}
|v_{a}(u_h)|_{S_a} \leq h^{\frac12} \gamma_{GLS}^{\frac12} \sup_{x \in \Omega} \mathcal{L^*} e^{-\eta}
\|u_h\|_V + \sup_{x \in \Omega} e^{-\eta} |u_h|_{S_a} + |\delta(e^{-\eta} u_h)|_{S_a} \\
\leq (C_{\gamma\beta\sigma\eta} h^{\frac12}  \sup_{x \in \Omega}
e^{-\eta} + \Pi(h)) \|u_h\|_V +\sup_{x \in \Omega}
e^{-\eta}
|u_h|_{S_a}.
\end{multline*}
We then use an inequality similar to \eqref{LstaraboundbySp}, but this time
adding
the boundary penalty term that is included in the stabilization in
formulation \eqref{FEM} (see \eqref{stab_ops_bound}.)
\begin{equation}\label{SaboundbySp}
|u_h|_{S_a} \leq  |u_h|_{S_p} + C_{\beta\gamma}h^{\frac12}
(2 \|\sigma\|_{L^\infty(\Omega)}  + \|\nabla \cdot \beta
\|_{L^\infty(\Omega)})\|u_h\|_V +  \gamma_{bc}^{\frac12} \||\beta
\cdot n|^{\frac12} u_h\|_{\partial \Omega}.
\end{equation}
Note that the boundary contribution can not be controlled by
$|u_h|_{S_p}$ as one would like, but must be controlled using the
$V$-norm. This
adds an $O(\gamma_{bc}^{\frac12})$ contribution to the constant in front of $\|u_h\|_V$.
\begin{equation}\label{SaboundbySpfinal}
|u_h|_{S_a} \leq |u_h|_{S_p} + (\gamma_{bc}^{\frac12} + C_{\beta\gamma}(h^{\frac12}
(2 \|\sigma\|_{L^\infty(\Omega)}  + \|\nabla \cdot \beta
\|_{L^\infty(\Omega)}))\|u_h\|_V.
\end{equation}

The proof of \eqref{stab_v_bound1} is similar, but here the stronger
adjoint boundary penalty can control the boundary term, leading to
\[
|z_h|_{S_p} \leq |z_h|_{S_a} + (C_{\beta\gamma}(h^{\frac12}
(2 \|\sigma\|_{L^\infty(\Omega)}  + \|\nabla \cdot \beta
\|_{L^\infty(\Omega)}))\|z_h\|_W.
\]
We conclude that the inequalities \eqref{stab_v_bound2} and
\eqref{stab_v_bound1} hold with 
$$\tilde c_\eta \ge \sup_{x \in \Omega}
e^{-\eta} + \Pi(h) \mbox{ and }\breve \epsilon(h) \ge C_{\beta\sigma\gamma\eta}
h^{\frac12} + \sup_{x \in \Omega}
e^{-\eta}\gamma_{bc}^{\frac12}
$$
The continuity \eqref{cont1} is
immediate by
integration by parts and Cauchy-Schwarz inequality,
\begin{align*}
a_h(v - \pi_V v, x_h) &=   (u - \pi_V u, \mathcal{L}^* x_h)+
\int_{\partial \Omega} (\beta \cdot n)( v - \pi_V v) x_h ~\mbox{d}s\\
&\leq C_\gamma
\|u - \pi_V u\|_+ (|x_h|_{S_a} + \|x_h\|_W).
\end{align*}
\end{proof}

\begin{remark}
Note that for the GLS-method $\tilde \epsilon(h)=0$ in
\eqref{forward_stability} and \eqref{adjoint_stability} indicating
that the scheme is unconditionally stable. This follows from the fact that the
whole residual is considered in the stabilization term. This nice
feature however only holds under exact quadrature. When the integrals are
approximated, the quadrature error once again gives rise to
oscillation terms from data that introduces a non-zero
contribution to $\tilde \epsilon(h)$.
\end{remark}
\subsection{Continuous interior penalty}\label{subsec:CIP}
In this case also $W_h = V_h:= X_h^k \cap H^1(\Omega)$, but the stabilization added to the standard Galerkin
formulation is a penalty on the jump of the gradient over element
faces \cite{DD76,BH04}.
The key observation is that the following discrete approximation
result holds for $\gamma_{CIP}$ large enough (see \cite{Bu05, BFH06})
\begin{equation}\label{oswald_int}
\|h^{\frac12} |\beta_h|^{-\frac12}(\beta_h \cdot \nabla u_h - I_{os} \beta_h \cdot \nabla
u_h)\|^2 \leq  s_{CIP}(u_h,u_h).
\end{equation}
Here $\beta_h$ is some piecewise affine interpolant of the velocity
vector field $\beta$ and $I_{os}$ is the quasi-interpolation operator defined in each
node of the mesh as a straight average of the function values from
triangles sharing that node,
\[
(I_{os} \beta_h \cdot \nabla u_h)(x_i) = N_i^{-1} \sum_{\{K : x_i \in K\}} (\beta_h \cdot \nabla u_h)(x_i)\vert_K,
\]
with $N_i := \mbox{card} \{K : x_i \in K\}$.
Stability is then a
consequence of the following lemma:
\begin{lemma}\label{stab_cont}
The following inequalities hold.
\begin{equation}\label{discrete_interp}
\inf_{v_h \in V_h} \|h^{\frac12}(\mathcal{L} u_h - v_h)\| \leq C_{\gamma\beta}
s_{CIP}(u_h,u_h)^{\frac12} + \epsilon_{CIP}(h) \|u_h\|
\end{equation}
and
\begin{equation}\label{disc_adjoint}
\inf_{w_h \in W_h} \|h^{\frac12}(\mathcal{L} ^*z_h - w_h)\| \leq C_{\gamma\beta}
s_{CIP}(z_h,z_h)^{\frac12} + \epsilon_{CIP}(h) \|z_h\|,
\end{equation}
with
$
\epsilon_{CIP}(h)  \sim  h^{\frac32} (\|\beta\|_{W^{2,\infty}(\Omega)} + c_{dc,0,\sigma}).
$
\end{lemma}
\begin{proof}
Since the proofs of the two results are similar we only detail the
arguments for \eqref{discrete_interp}.
First note that 
\begin{multline*}
\inf_{v_h \in V_h} \|h^{\frac12} (\mathcal{L} u_h - v_h)\| \leq
\|h^{\frac12} (i_h \beta \cdot
\nabla u_h - I_{os}( i_h \beta \cdot
\nabla u_h)) \| \\
+ h^\frac12 \|\beta -i_h\beta\|_{L^\infty(\Omega)} \|\nabla
u_h\| + h^\frac12\|\sigma u_h - i_h (\sigma u_h)\|.
\end{multline*}
Using \eqref{oswald_int}, interpolation in $L^\infty$, an inverse inequality and the
discrete commutator property \eqref{discrete_commutator} we conclude
\[
\inf_{v_h \in V_h} \|h^{\frac12} (\mathcal{L} u_h - v_h)\| \leq
C_{\beta\gamma} s_{CIP}(u_h,u_h)^{\frac12}\\
+ h^{\frac32} (\|\beta\|_{W^{2,\infty}(\Omega)} + c_{dc,0,\sigma}) \|u_h\|.
\]
\end{proof}
For the CIP-method we choose the $\pi_V$ and $\pi_W$ as the $L^2$-projection in order to exploit
orthogonality to ``filter'' the element residual. Observe that if
$u\in H^{\frac32+\varepsilon}(\Omega)$, $\varepsilon>0$ then
$s_{CIP}(u,\cdot)=0$. The consistencies \eqref{galortho1}
and \eqref{galortho2} hold from the consistency of
\eqref{discrete_bilin}.
The approximation result \eqref{approx1}, with $r=k+\tfrac12$ is a consequence of
standard results for the CIP-method (see for instance \cite{BFH06}.)
We now prove that the remaining assumptions for Proposition \ref{standard_conv} and Theorem
\ref{stab_conv} hold.
\begin{proposition}(Satisfaction of assumptions for
  \eqref{standardFEM} with CIP)\label{assumpCIP}
Let the bilinear forms of \eqref{standardFEM} be defined by
\eqref{discrete_bilin} and \eqref{jumpstab1}. Let $\gamma_{bc}\ge1$. Then \eqref{galortho1}--
\eqref{cont2} are
satisfied, with $\epsilon(h) \sim h^{\frac12}$.
\end{proposition}
\begin{proof}
To prove the stability \eqref{strong_bound_va} take $v_a =
\pi_V(e^{-\eta} u_h)$ and use Lemmas \ref{basic_stability}, the
orthogonality of the $L^2$-projection and
\ref{stab_cont} to obtain
\begin{multline}\label{stand_stab_CIP}
 a_h(u_h,\pi_V (e^{-\eta} u_h)) =  a_h(u_h, e^{-\eta} u_h)
 -(\mathcal{L} u_h - w_h,\delta (e^{-\eta} u_h)) \\ \ge -C_{\gamma} |u_h|_{S_p} \|\delta (e^{-\eta} u_h)\|_+- \epsilon_{CIP}(h) \|u_h\| \|h^{-\frac12}\delta (e^{-\eta} u_h)\|\\+ \frac12 \int_{\partial \Omega} (\beta \cdot n) u_h^2
e^{-\eta} ~\mbox{d}s+ 
\frac12 \|u_h e^{-\frac{\eta}{2}} \|^2.
\end{multline}
We also observe that for the stabilization 
\begin{equation}\label{s_stab_CIP}
s_p(u_h,\pi_V (e^{-\eta} u_h)) \ge s_p(u_h,u_h e^{-\eta} ) -
|u_h|_{S_p} |\delta (e^{-\eta} u_h) |_{S_p}. 
\end{equation}
Now observe that, since the jump of  $\nabla e^{-\eta}$ is zero we
have, using \eqref{stand_stab_CIP} and \eqref{s_stab_CIP}
\begin{multline*}
\frac12 \inf_{x \in \Omega} e^{-\eta} (\|u_h\|_V^2 + |u_h|^2_{S_p})
\leq \frac12 \|u_h e^{-\frac{\eta}{2}} \|^2 + \frac12 \int_{\partial
  \Omega} (\beta \cdot n) u_h^2 e^{-\eta} ~\mbox{d}s+  s_p(u_h,u_h
e^{-\eta} )\\
\leq  a_h(u_h,\pi_V (e^{-\eta} u_h)) + s_p(u_h,\pi_V (e^{-\eta} u_h))
\\|u_h|_{S_p} (C_{\gamma} \|\delta
(e^{-\eta} u_h)\|_++|\delta (e^{-\eta} u_h) |_{S_p})+ \epsilon_{CIP}(h) \|u_h\| \|h^{-\frac12}\delta
(e^{-\eta} u_h)\|
\end{multline*}
Using Lemma \ref{interp_lem} we deduce that
\eqref{strong_bound_va} holds with 
$$c_s = \frac12 \inf_{x \in \Omega}
e^{-\eta} \mbox{ and }
 \epsilon(h) \ge \Pi(h) (C_{\gamma} +   \epsilon_{CIP}(h)).
$$
For \eqref{stab_v_bound3} only the stabilization part differs from the
GLS case. Since the
jump of $\nabla e^{-\eta}$ is zero we immediately get
\[
|v_{a}(u_h)|_{S_p} \leq \sup_{x \in \Omega} e^{-\eta} |u_h|_{S_p} +
|\delta(e^{-\eta} u_h) |_{S_p} \leq \sup_{x \in \Omega} e^{-\eta}
|u_h|_{S_p} + \Pi(h) \|u\|_V
\]
and hence $c_\eta \ge \sup_{x \in \Omega} e^{-\eta}+\Pi(h)$.
The continuity \eqref{cont2} follows by observing that by
\eqref{stab_cont} there holds

\begin{multline} \label{CIPcont}
a_h(v- \pi_V v, x_h) =\inf_{w_h \in V_h}  (v - \pi_V v, \mathcal{L}^* x_h - w_h) +
\int_{\partial \Omega} (\beta \cdot n) (v- \pi_V v) x_h ~\mbox{d}s
\\
\leq \|v - \pi_V v\|_+(C_\gamma |x_h|_{S_p} + (C_\beta h^{\frac12}
\epsilon_{CIP}(h) + 1)\|x_h\|_V ).
\end{multline}
Where we observe that the boundary part must be controlled using the norm
$\|\cdot\|_V$.
\end{proof}
\begin{proposition}(Satisfaction of assumptions for
  \eqref{FEM} with CIP)\label{assumpCIP_new}
 Let the bilinear forms of \eqref{FEM} be defined by
\eqref{discrete_bilin} and \eqref{jumpstab1} for both
$s_p(\cdot,\cdot)$ and $s_a(\cdot,\cdot)$, together with the respective
boundary penalty terms of \eqref{stab_ops_bound}. Then the inequalities \eqref{galortho2}--\eqref{cont1} hold with $\tilde \epsilon(h)
\sim h^2$.
\end{proposition}
\begin{proof}
Starting from \eqref{stand_stab_CIP} with $v_a(u_h) := \pi_W(e^{-\eta} u_h)$ we have using Lemma \ref{interp_lem},
\begin{multline}\label{new_stab_CIP}
 \frac12 \|u_h e^{-\frac{\eta}{2}} \|^2 + \frac12 \int_{\partial \Omega} |\beta \cdot n| u_h^2
e^{-\eta} ~\mbox{d}s \leq a_h(u_h,\pi_W (e^{-\eta} u_h)) - \int_{\partial \Omega} (\beta \cdot n)_- u_h^2
e^{-\eta} ~\mbox{d}s \\
 + (C_{\gamma} |u_h|_{S_p}+ \epsilon_{CIP}(h) \|u_h\| ) \Pi(h) \|u_h\|
 \\
\leq  a_h(u_h,\pi_W (e^{-\eta} u_h)) + (\gamma_{bc}^{-\frac12} \sup_{x
  \in \Omega} e^{-\eta} + C^2_{\gamma} \sup_{x
  \in \Omega} e^{\eta}\Pi(h)^2) |u_h|_{S_p} \\+ \left(\frac14 \inf_{x \in \Omega}
e^{-\eta} + \epsilon_{CIP}(h) \Pi(h)\right) \|u_h\|^2.
\end{multline}
The last inequlaity is due to an arithmetic-geometric inequality. 
 Hence we see that \eqref{forward_stability} holds with 
$\tilde \epsilon(h)= \epsilon_{CIP}(h) \Pi(h) \sim h^{2}$ and $$\tilde
c_s = \frac14 \inf_{x \in \Omega} e^{-\eta},\quad \tilde c_\eta \ge C^2_\gamma \sup_{x
  \in \Omega} e^{\eta}  \Pi(h)^2 + \gamma_{bc}^{-1} \sup_{x
  \in \Omega} e^{-\eta} 
$$
The inequality \eqref{adjoint_stability} is proved
similarly as in the GLS case, taking this time $v_{a*}(z_h) := -\pi_V
(e^{-\eta} z_h)$, with $\pi_V$ the $L^2$-projection and
using the second inequality of Lemma \ref{basic_stability} and Lemma
\ref{interp_lem} after
integration by parts. 
\begin{multline*}
a_h(\delta (e^{-\eta} z_h),z_h) = \int_{\partial \Omega} (\beta \cdot
n) \delta (e^{-\eta} z_h) z_h ~\mbox{d}s + \inf_{w_h \in W_h} (\delta (e^{-\eta} z_h),
\mathcal{L^*} z_h - w_h) \\
\leq C \|\delta (e^{-\eta} z_h)\|_+ (|z_h|_{S_a} + \epsilon_{CIP}(h)
\|z_h\|) \leq C \Pi(h)
\|z_h\| (|z_h|_{S_a} + \epsilon_{CIP}(h) \|z_h\|).
\end{multline*}
Then we conclude as before.
For the stabilities \eqref{stab_v_bound2}  and \eqref{stab_v_bound1} we
proceed as in Proposition \ref{assumpGLS} and we only detail
the second inequality of \eqref{stab_v_bound2}. When using the CIP-method the primal and adjoint stabilization terms only
differ in the boundary contributions therefore, by symmetry, the second
inequality of \eqref{stab_v_bound1} follows identically.
Since the jump of $\nabla e^{-\eta}$ is zero we get
\begin{equation}
|v_{a}(u_h)|_{S_a} \leq \sup_{x \in \Omega} e^{-\eta} |u_h|_{S_p} +
|\delta(e^{-\eta} u_h)|_{S_p}+
\gamma_{bc}^{\frac12}\| |\beta \cdot n|^{\frac12} v_{a}(u_h)\|_{\partial \Omega}.
\end{equation}
The boundary penalty term is bounded using that by adding and
subtracting $e^{-\eta} u_h$, using a triangle inequality and the arguments of Lemma \ref{interp_lem} to obtain
\begin{multline}\label{critical_step}
\gamma_{bc}^{\frac12}\| |\beta \cdot n|^{\frac12} v_{a}(u_h)\|_{\partial \Omega} \leq \Pi(h) \|u_h\|_V + \gamma^{\frac12}_{bc}\||\beta|_+^{\frac12} u_h
e^{-{\eta}}\|_{\partial \Omega} \\
\leq (\Pi(h) + \gamma_{bc}^{\frac12}\sup_{x \in
  \Omega} e^{-\eta}) \|u_h\|_V
\end{multline}
Therefore \eqref{stab_v_bound2}  and \eqref{stab_v_bound1} hold with
\begin{equation}\label{constants}
\tilde c_\eta \ge  \sup_{x \in \Omega} e^{-\eta} + \Pi(h)\mbox{ and }\breve
\epsilon(h) \ge
\gamma_{bc}^{\frac12} \sup_{x \in \Omega}
e^{-\eta} + 2 \Pi(h).
\end{equation}
The proof of continuity \eqref{cont1} follows as in \eqref{CIPcont}.
\end{proof}
\subsection{The discontinuous Galerkin method}\label{subsec:DG}
In the case where discontinuous elements are used, i.e. $V_h = W_h :=
X_h^k$ the analysis is simplified by the fact that $\beta_h
\cdot \nabla u_h\in V_h$. Here we let $\pi_V$ and $\pi_W$ denote the element wise
$L^2$-projection onto $X_h^k$. The analysis is essentially the same as
for the CIP-method and when appropriate we will refer to the
previous analysis. Thanks to the local
character of the DG method the results hold without assuming any
quasi regularity of the meshes.
The consistency results \eqref{galortho1}
and \eqref{galortho2} are standard as well as the approximation result
\eqref{approx1}, with $r=k+\tfrac12$ (see \cite{EG06}). As before we collect the proofs of the remaining assumption 
in a proposition.
\begin{proposition}(Satisfaction of assumptions for
  \eqref{standardFEM} with DG)\label{DGassump}
Let the bilinear forms of \eqref{standardFEM} be defined by
\eqref{discrete_bilin} and \eqref{DGstab}. Then \eqref{galortho1}--\eqref{cont2} are
satisfied with $\epsilon(h) \sim h^{\frac12}$.
\end{proposition}
\begin{proof}
Let $i_h \beta \in X_h^1$ be the Lagrange interpolant of $\beta$ with and $\pi_0 \sigma \in
X_h^0$ the projection of $\sigma$ on piecewise constant functions.
 For \eqref{strong_bound_va} take $v_a := \pi_V (
e^{-\eta} u_h)$, use $L^2$-orthogonality and apply Lemma
\ref{basic_stability} to obtain for $\gamma_{bc}$ large enough,
\begin{multline}\label{bas_stab_DG}
 a_h(u_h,\pi_V (e^{-\eta} u_h)) + s_p(u_h,\pi_V (e^{-\eta} u_h))=  a_h(u_h, e^{-\eta} u_h) + s_p(u_h,e^{-\eta} u_h)\\- a_h(u_h,\delta (e^{-\eta} u_h))- s_p(u_h,\delta(e^{-\eta} u_h))
 \ge ((i_h \beta - \beta) \cdot \nabla u_h + (\pi_0 \sigma - \sigma)
 u_h,  \delta (e^{-\eta} u_h))\\
-2 \sum_{K \in \mathcal{T}_h} \left<|\beta \cdot n| |[u_h]|, (1 +
  \gamma_{DG})|\delta (e^{-\eta} u_h)|\right>_{\partial K \setminus \partial \Omega}
+ \frac12 \inf_{x \in \Omega}
e^{-\eta} (\|u_h\|^2_V + |u_h|^2_{S_p})\\
\ge -|u_h|_{S_p} C_{\gamma} \|\delta (e^{-\eta} u_h)\|_+ -  \epsilon_{DG}(h) \|u_h\| \|h^{-\frac12}\delta (e^{-\eta} u_h)\|\\
+  \frac12 \inf_{x \in \Omega}
e^{-\eta} (\|u_h\|^2_V + |u_h|^2_{S_p})
\end{multline}
where 
$$\epsilon_{DG}(h) = h^{-\frac12} \|i_h \beta -
\beta\|_{L^\infty(\Omega)} + h^{\frac12} \|\pi_0 \sigma -
\sigma\|_{L^\infty(\Omega)}  \sim (\|\beta\|_{W^{2,\infty}(\Omega)} +
\|\sigma\|_{W^{1,\infty}(\Omega)}) )h^{\frac32}.$$ 
It follows that \eqref{strong_bound_va} holds with
$c_s = \tfrac12 \inf_{x \in \Omega}
e^{-\eta}$ and $\epsilon(h) \ge (C_\gamma + \epsilon_{DG}(h))
\Pi(h)$. The proof of \eqref{stab_v_bound3} is
analogous with the CIP-case with similar constants.
Considering finally the continuity \eqref{cont2} we have after an
integration by parts
\begin{multline}\label{cont_DG}
a_h(v - \pi_V v, x_h) = (v - \pi_V v, \mathcal{L}^* x_h) + \frac12
\sum_K \left<(\beta \cdot n) \{v - \pi_V v \}, [x_h] \right>_{\partial
  K \setminus \partial \Omega}\\
+ \left<(\beta \cdot n) (v - \pi_V v ),x_h \right>_{\partial
  \Omega}\\
= (v - \pi_V v, (i_h \beta - \beta) \nabla x_h + (\sigma - \pi_0
\sigma) x_h)  + \frac12
\sum_K \left<(\beta \cdot n) \{v - \pi_V v \}, [x_h] \right>_{\partial
  K \setminus \partial \Omega}\\
+ \left<(\beta \cdot n) (v - \pi_V v ),x_h \right>_{\partial
  \Omega}\\
\leq\|v - \pi_V v\|_+ (C_{\gamma} |x_h|_{S_a} + (C_{\beta} h^{\frac12}
\epsilon_{DG}(h) + 1)\|x_h\|_V).
\end{multline}
\end{proof}
\begin{proposition}
(Satisfaction of assumptions for
  \eqref{FEM} with DG)\label{DGassump_new}
 Let the bilinear forms of \eqref{FEM} be defined by
\eqref{discrete_bilin} and \eqref{DGstab} for both
$s_p(\cdot,\cdot)$ and $s_a(\cdot,\cdot)$ together with the respective
boundary penalty terms of \eqref{stab_ops_bound}. Then the
inequalities \eqref{galortho2}--\eqref{cont1} hold with $\tilde \epsilon(h) \sim h^2$.
\end{proposition}
\begin{proof}
The stability \eqref{forward_stability} and \eqref{adjoint_stability}
follows by taking $v_a := \pi_W (
e^{-\eta} u_h)$ and $v_{a*} := -\pi_V (
e^{-\eta} z_h)$, using \eqref{bas_stab_DG} and the manipulations of Proposition
\ref{assumpCIP_new}. The proof of the inequalities \eqref{stab_v_bound2} and
\eqref{stab_v_bound1} use the same techniques as the corresponding
results for the CIP-method and result in similar constants. Finally
\eqref{cont1} follows from \eqref{cont_DG}.
\end{proof}

\subsection{Convergence of the error in the streamline derivative}
As already mentioned the natural norm for the above analysis would
include the $L^2$-norm of the $h^\frac12$-weighted streamline
derivative. Given the results of the previous section it is
straightforward to prove optimal convergence of the streamline
derivative both for \eqref{standardFEM} and \eqref{FEM}.
We only give the result for the
method \eqref{FEM} below. The proof of the result for \eqref{standardFEM} is identical.
\begin{proposition}
Let $u_h,z_h$ be the solution of \eqref{FEM} with bilinear form
\eqref{discrete_bilin} stabilized with one of the methods presented in
Sections \ref{subsec:GLS} -- \ref{subsec:DG}. Assume that the
conditions of Theorem \ref{stab_conv} are satisfied. Then there holds
\[
\|\beta \cdot \nabla (u - u_h)\|_h \leq C_{\eta\beta\sigma\gamma}
h^k |u|_{H^{k+1}(\Omega)}.
\]
\end{proposition}
\begin{proof}
First consider the GLS-method. Add and subtract $\sigma(u - u_h)$ inside
the streamline derivative norm and use a triangle inequality to
obtain, using the previously obtained error estimates,
\[
\|\beta \cdot \nabla (u - u_h)\| \leq C_\gamma \|\beta\|_{L^\infty}^{-\frac12} h^{-\frac12} (|u-u_h|_{S_p} +
\|\sigma\|_{L^\infty(\Omega)} h^{\frac12} \|u-u_h\|) \leq C_{\gamma\beta\sigma} h^k |u|_{H^{k+1}(\Omega)}.
\]
For the CIP method we may write $\xi_h := \pi_V u - u_h$, where
$\pi_V$ is any interpolation operator with optimal approximation properties, and note that
by Galerkin orthogonality, interpolation in $L^\infty$, and inverse inequalities, we have
\begin{multline*}
\|\beta \cdot \nabla (u - u_h)\|^2 = (\beta \cdot \nabla (u -
u_h),\beta_h \cdot \nabla \xi_h - 
I_{os} \beta_h \cdot \nabla \xi_h ) -
(\sigma (u - u_h) , I_{os} \beta_h \cdot \nabla \xi_h) \\ - s_a(z_h, I_{os} \beta_h \cdot \nabla \xi_h)\\+ (\beta \cdot
\nabla (u - u_h),  (\beta -\beta_h) \cdot \nabla \xi_h) -  (\beta \cdot
\nabla (u - u_h), \beta \cdot \nabla (\pi_V u - u))\\
\leq C_\gamma\|\beta \cdot \nabla (u - u_h)\| (h^{-\frac12} |\xi_h|_{S_p}
+ \|\beta\|_{W^{1,\infty}(\Omega)} \|\xi_h\| + \|\beta \cdot \nabla
(u - \pi_V u)\|) \\
+  (C_{\gamma\beta\sigma} h^{-\frac12}  |z_h|_{S_a} + \|\sigma\|_{L^\infty(\Omega)} \|u-u_h\|)\| \beta_h \cdot \nabla \xi_h\|.
\end{multline*}
Here we have used the $L^2$-stability of the interpolation operator
$I_{os}$ and the inequality
\[
|s_a(z_h, I_{os} \beta_h \cdot \nabla \xi_h)| \leq |z_h|_{S_a}
C_{\gamma\beta\sigma} h^{-\frac12} \|\beta_h \cdot \nabla \xi_h\|.
\]
Observing that 
\[
\| \beta_h \cdot \nabla \xi_h\| \leq C \|\beta\|_{W^{1,\infty}(\Omega)}
\|\xi_h\| + \|\beta \cdot \nabla (u - u_h)\| +  \|\beta \cdot \nabla
(u - \pi_V u)\|
\]
and using suitable arithmetic-geometric inequalitites to absorb
factors $\|\beta \cdot \nabla (u - u_h)\|$ in the left hand side we
conclude that
\begin{multline*}
\|\beta \cdot \nabla (u - u_h)\|^2  \leq C_{\gamma\beta\sigma} \Bigl(h^{-1} |\xi_h|^2_{S_p}
+ \|\xi_h\|^2 + \|u - u_h\|^2\\+ h^{-1} |z_h|^2_{S_a}+ \|\beta \cdot \nabla
(u - \pi_V u)\|^2 \Bigr)\leq C_{\gamma\beta\sigma} h^{2k} |u|^2_{H^{k+1}(\Omega)}.
\end{multline*}
The last inequality is a consequence of the estimate
\[
\|u - u_h\|_V + |u-u_h|_{S_p} + |z_h|_{S_a} \leq C_{\gamma\beta\sigma} h^{k+\frac12} |u|_{H^{k+1}(\Omega)}
\]
of Theorem \ref{stab_conv} and standard approximation results on $\|u-
\pi_V u\|$ and $\|\beta \cdot \nabla
(u - \pi_V u)\|$.
The proof for the discontinuous Galerkin method is similar and left to
the reader.
\end{proof}

\subsection{The data assimilation case}
The aim of the methods presented in \cite{part1}, is to introduce a
framework where also ill-posed problems such as those arising in
inverse problems, or data assimilation problems can
be included, without modifying the method. We will therefore in this section discuss the case where
 {\emph{data is given on the outflow boundary}} in equation
 \eqref{model_problem} as a model case of data assimilation. By the reversibility of the transport equation
 under our assumptions on $\beta$ this problem is not ill-posed on the
continuous level. 
However on the discrete level methods based on upwinding are
likely to experience difficulties. Since our framework does not rely
neither on upwinding nor on coercivity, this case can be included with only minor modifications
in the formulations without any loss of
stability. Consider the problem \eqref{model_problem} with the
boundary condition $u = g$ on $\partial \Omega_+$.  Let the
formulation \eqref{FEM} be defined by the  bilinear form \eqref{discrete_bilin}
and the stabilization term $s_p(\cdot,\cdot)$ for $X=GLS, CIP, DG$, 
\begin{equation}\label{assim_stab}
s_p(u_h,v_h) := s_{p,X}(u_h,v_h) + s_{bc,+}(u_h,v_h).
\end{equation}
The term $s_a(\cdot,\cdot)$ is unchanged. The data assimilation problem then typically consists in finding $u|_{\partial
  \Omega_{-}}$, which amounts to solving the backward transport equation. 
Observe that the boundary penalty for the primal equation now acts on
the outflow boundary.
The stabilization may then be chosen as any of the three methods
considered in Section \ref{subsec:GLS}--\ref{subsec:DG} and Theorem
\ref{stab_conv} holds under the same conditions as before, but the
stability will be given by a different weight function. Once the functions $v_{a}$ and
$v_{a*}$ have been identified the rest of the analysis is identical to
that of Sections  \ref{subsec:GLS}--\ref{subsec:DG}. We recall the
following inequalities from Lemma \ref{basic_stability}.
\begin{lemma}\label{daat_assim_stability}
For the bilinear form \eqref{discrete_bilin} there holds, for all
$\eta \in W^{1,\infty}(\Omega)$
\[
a_h(u_h,-e^{\eta} u_h) = -\frac12 \int_{\partial \Omega} (\beta \cdot
n) u_h^2 e^{\eta} ~\mbox{d}s + \int_{\Omega} u_h^2 \left(\frac12 \beta
\cdot \nabla \eta + \frac12 \nabla \cdot \beta - \sigma\right) e^{\eta}~\mbox{d}x,
\]
\[
a_h( e^{\eta} z_h, z_h) = \frac12 \int_{\partial \Omega} (\beta \cdot
n) z_h^2 e^{\eta} ~\mbox{d}s + \int_{\Omega} z_h^2 \left(\frac12 \beta
\cdot \nabla \eta + \frac12 \nabla \cdot \beta - \sigma\right) e^{\eta}~\mbox{d}x.
\]
\end{lemma}
{
It follows that apart from the form of the exponential dependencies in the
constants nothing changes for the method \eqref{FEM}. The situation is
different for method \eqref{standardFEM}, since here the same test
function must be used in the forms $a_h(\cdot,\cdot)$ and
$s_p(\cdot,\cdot)$. We see that the choice $v_{a}(u_h):=-\pi_V (e^{\eta}
u_h)$ is necessary in $a_h(\cdot,\cdot)$, however due to the least
squares character of $s_p(\cdot,\cdot)$, the term can never have a
stabilizing effect for positive stabilization parameter when this
weight function is used. If instead
the stabilization parameters in \eqref{standardFEM} are chosen
\emph{negative} it is straightforward to show that the assumptions for
Proposition \ref{standard_conv} hold. This correpsonds to using
downwind fluxes instead of upwind fluxes. For more general problems
however, data are provided at some points along the
characteristics and it is therefore not possible for any given point
in the domain to decide whether the data will arrive from the upwind
or the downwind side unless the characteristic equations are solved for each
given data. Therefore the strategy of changing the sign of the
stabilization parameter inside the domain to match the location of
given data is not so attractive. In contrast the method \eqref{FEM}
does not use the flow direction for stability and can therefore be
applied in a much wider context, without tuning the stabilization parameters.}
\section{Numerical examples}\label{numerics}
Here we will give some simple numerical examples illustrating the
above theory. All computations were made using Freefem++ \cite{freefem}. We will only consider the CIP-method and compare
the results obtained by \eqref{standardFEM} with those of \eqref{FEM}
and in some cases with the
standard Galerkin method. We use an exact solution from \cite{CD11} adapted for
the case of vanishing viscosity with some different velocity
fields. We consider pure transport on conservation form and with a
non-solenoidal velocity field,
\begin{equation}\label{comptransp}
\nabla \cdot (\beta u) = f \quad \mbox{ on } \Omega.
\end{equation}
Three different velocity fields will be used,
\begin{equation}\label{beta1}
\beta_1:= \left(\begin{array}{c} -(x+1)^4+y \\ -8(y-x) \end{array} \right),
\end{equation}
\begin{equation}\label{beta2}
\beta_2:= -100 \left(\begin{array}{c} x+y \\ y-x \end{array} \right)
\end{equation}
or
\begin{equation}\label{badbeta}
\beta_3 = \left( \begin{array}{c}
  10  \,\mbox{arctan}(\frac{y-\frac12}{\varepsilon}) - \frac{x^2}{\varepsilon} \\[3mm]
\sin(x/\varepsilon) + \sin(y/\varepsilon)\end{array}\right).
\end{equation}
We will consider two different exact solutions, one smooth given by
\begin{equation}\label{exact_sol}
u(x,y) = 30x(1-x)y(1-y),
\end{equation}
obtained by choosing a suitable right-hand side $f$, and one non-smooth obtained by setting $f=0$, but introducing a
discontinuous function for the boundary data.
The smooth solution \eqref{exact_sol} satisfies homogeneous Dirichlet boundary conditions both
on the inflow and the outflow boundary and
has $\|u\| = 1$.
Unless otherwise stated, we use the stabilization parameters $\gamma_{CIP} =
0.01$ for piecewise affine approximation and $\gamma_{CIP}=0.001$ for
piecewise quadratic approximation. The boundary penalty term is taken
as $\gamma_{bc} = 0.5$ for \eqref{FEM} and $\gamma_{bc} = 1.0$ for
\eqref{standardFEM}. 

We have first 
considered the velocity field \eqref{beta1} and the solution \eqref{exact_sol}.  Note that $\inf_{x \in \Omega} \nabla
\cdot \beta_1 = -40$, making the problem strongly
noncoercive, since then $\sigma_0 = \tfrac12 \inf_{x \in \Omega} \nabla
\cdot \beta_1 = -20$. In our experience the standard Galerkin method performs relatively
well for the coercive case when approximating smooth solutions in two
space dimensions.
 As can be seen in figure \ref{smooth_plot}, this is not
the case here. Three contour plots are presented representing
computations using the standard Galerkin method, the method
\eqref{standardFEM} and \eqref{FEM} on a $64\times64$ unstructured
mesh. Note the oscillations that persist in the standard
Galerkin solution, in spite of the smoothness of the solution. On all the meshes we considered, up to $256 \times
256$ these oscillations remained, although their amplitude
decreased. This highlights the increased need of stabilization for
noncoercive problems.
\begin{figure}
{\centering
\includegraphics[width=5cm]{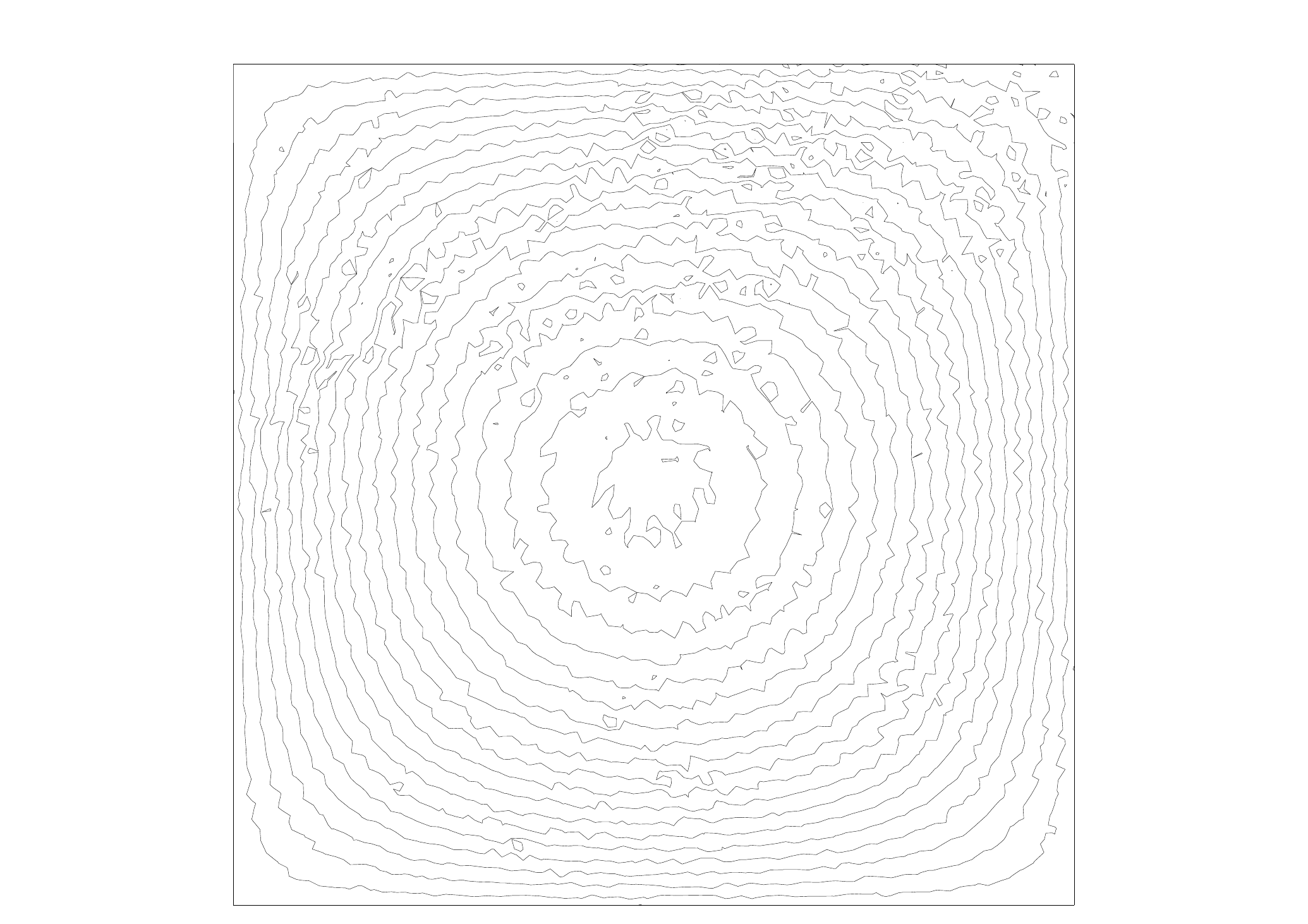}\hspace{-1.5cm}
\includegraphics[width=5cm]{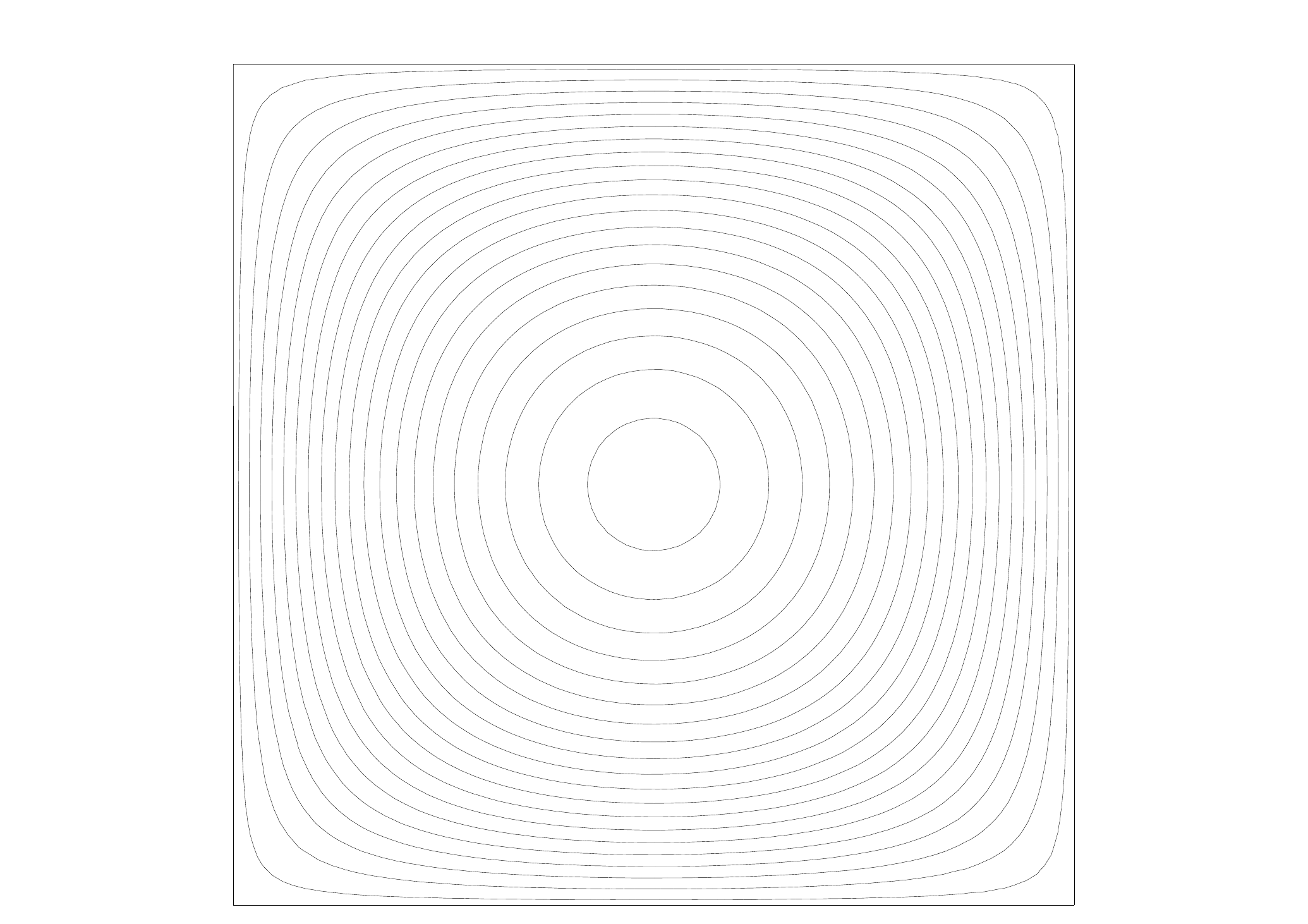}\hspace{-1.5cm}
\includegraphics[width=5cm]{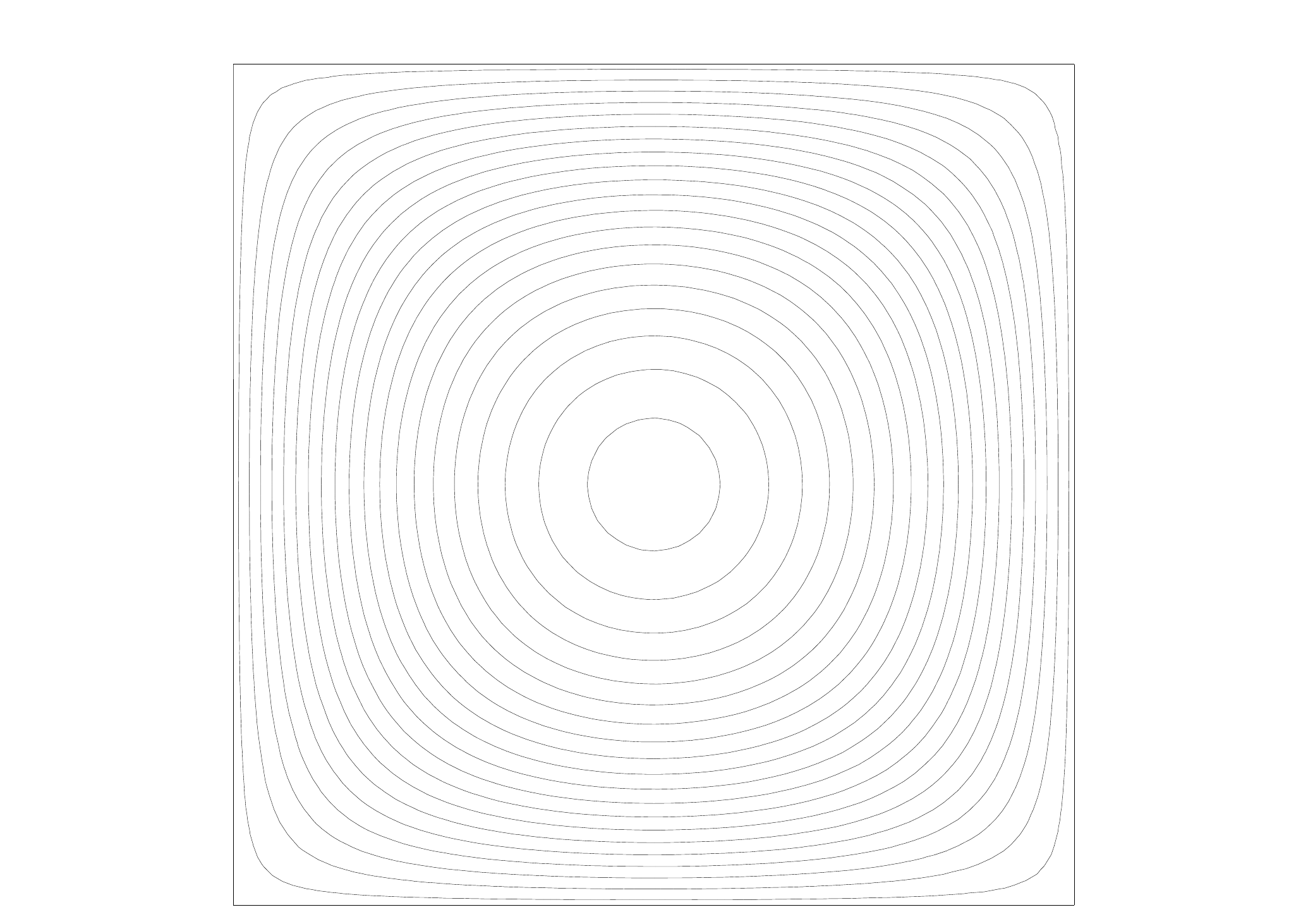}
}
\caption{Contour plots of approximations of the smooth solution \eqref{exact_sol}, $64\times 64$ mesh, affine approximation. From left
  to right, standard Galerkin, method \eqref{standardFEM}, method \eqref{FEM}.}\label{smooth_plot}
\end{figure}
In table \ref{SmoothP1} we present the errors in both
the $L^2$-norm and the streamline derivative norm, 
\begin{equation}\label{SDdef}
\|h^{\frac12} |\beta|^{-\frac12} \beta \cdot \nabla (u - u_h)\|,
\end{equation}
on $6$
consecutive unstructured meshes with $2^N$, $N=3,...,8$, elements on
each side and piecewise affine approximation. We note that the stabilized methods both have (and
sometimes exceed) the expected
convergence orders. Indeed the $L^2$-error converges as $O(h^{k+1})$
and the error in the streamline derivative \eqref{SDdef} as $O(h^{k+\frac12})$. As expected
the convergence of the standard Galerkin method is very uneven. It is
unclear if
the error in the streamline derivative converges at all. In table \ref{SmoothP2} the
same sequence of computations are reported using piecewise quadratic
elements. The stability of the standard Galerkin method is noticeably
improved. Nevertheless the errors of the stabilized methods are two
orders of magnitude smaller. The errors of formulation
\eqref{FEM} are slightly smaller than those of \eqref{standardFEM},
but on the other hand the former method uses twice as many degrees of
freedom as the latter.

Both methods \eqref{standardFEM} and \eqref{FEM} control
spurious oscillations in non-smooth exact solutions as can be seen in
figure \ref{rough_plot}, where the contour plots of a computation with 
non-smooth exact solution created by using the velocity field
\eqref{beta2} in \eqref{comptransp}, setting $f=0$ and the boundary
data equal one wherever $x>0.8$ and $y<0.5$ and zero elsewhere. 
\begin{figure}
{\centering
\includegraphics[width=5cm]{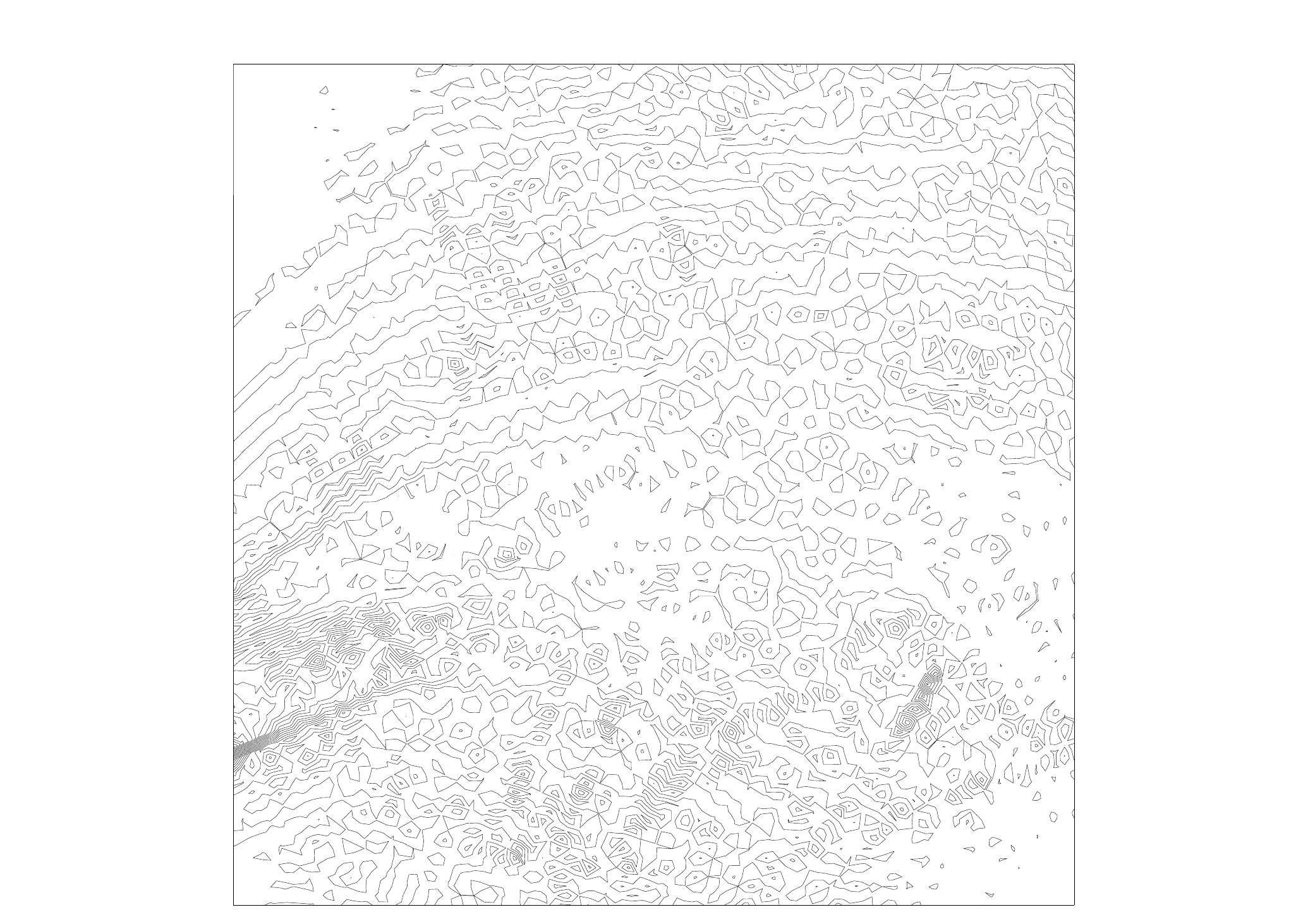}\hspace{-1.5cm}
\includegraphics[width=5cm]{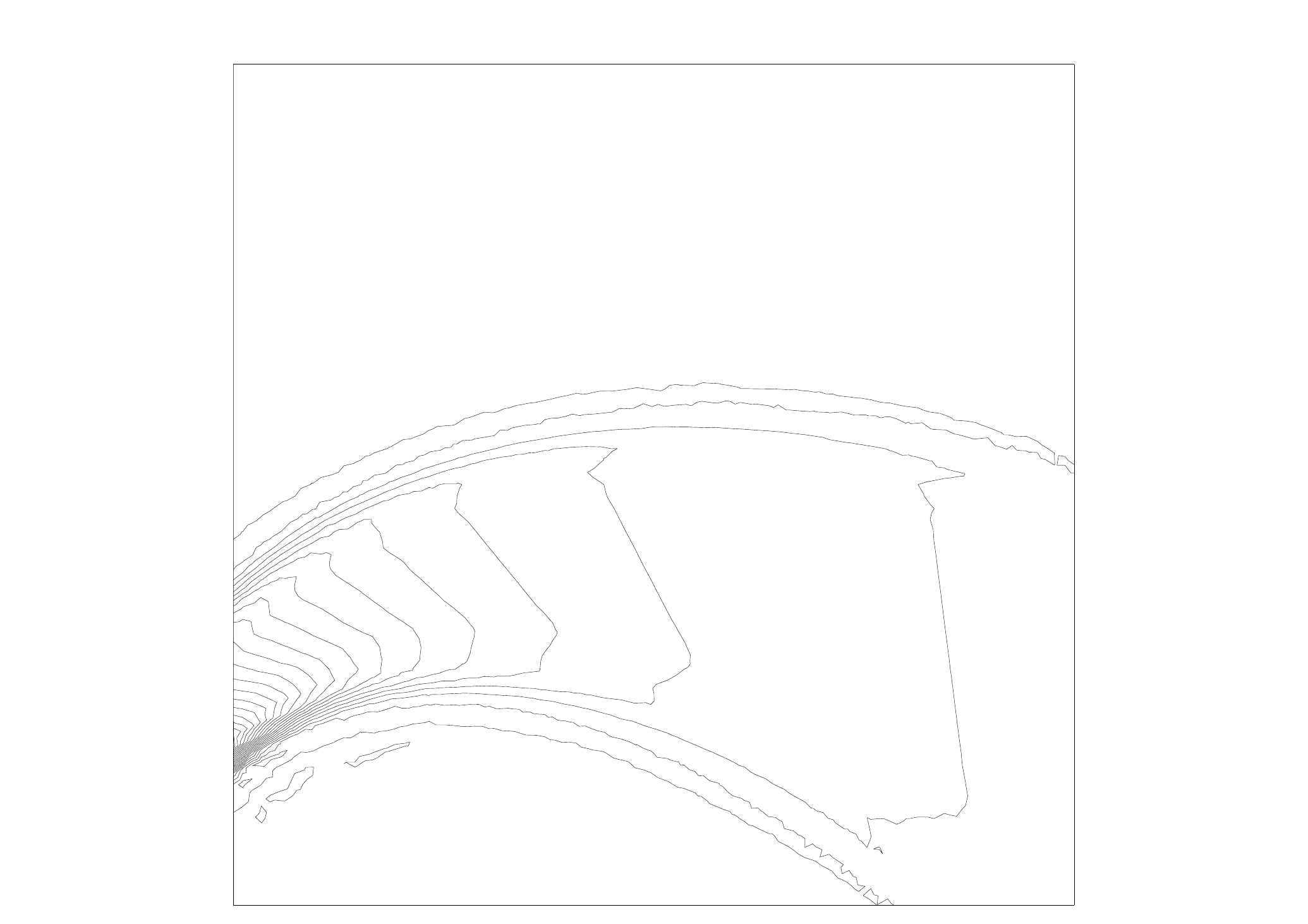}\hspace{-1.5cm}
\includegraphics[width=5cm]{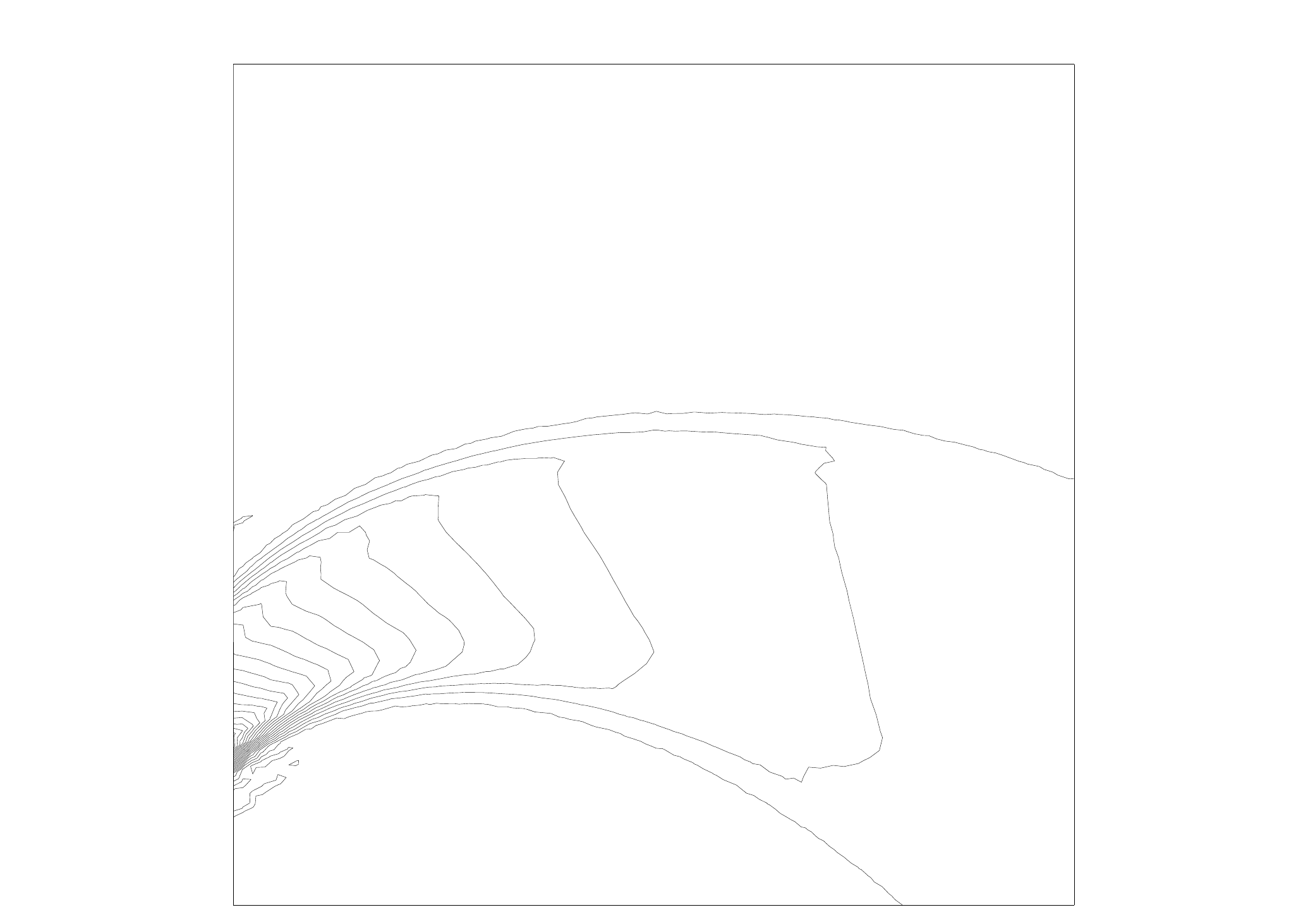}}
\caption{Discontinuous solution, $64\times 64$ mesh, affine approximation. From left
  to right, standard Galerkin, method \eqref{standardFEM}, method \eqref{FEM}.}\label{rough_plot}
\end{figure}
To show the increased
robustness of the formulation \eqref{FEM}, we propose to study the
problem \eqref{comptransp}, with the velocity field \eqref{badbeta}.
This velocity field is strictly speaking not covered by the analysis,
since for some values on $\varepsilon$ there may be points in the
domain where $\beta_3$ vanishes. Nevertheless the right hand side is chosen such that
the exact solution is given by \eqref{exact_sol}. We consider a fixed
$64\times 64$ unstructured mesh and vary $\varepsilon$, creating a series of
increasingly ill-posed problems where the divergence and the maximum
derivatives of $\beta$ behaves as
$-\tfrac{1}{\varepsilon}$.  The error in the
streamline derivative \eqref{SDdef} for varying $\varepsilon$ is
plotted in the left graphic of figure
\ref{stab_study}. It is fair to say that the method \eqref{FEM} (circle markers) outperforms
\eqref{standardFEM} (square markers). As $\varepsilon$ becomes small the error for the
approximations computed using \eqref{FEM} exhibits moderate growth
of order $O(\varepsilon^{-\frac13})$, but remain below $0.06$. Whereas
over half the approximations computed using \eqref{standardFEM} has an
error larger than $0.5$ and none below $0.1$. For $\varepsilon =
0.05$, the error is $120$ and the computed solution bears no
ressemblance to the exact one. { In the right plot of figure
\ref{stab_study} we study how the error
depends on the choice of the stabilization parameter $\gamma_{CIP}$.
We plot the error defined by \eqref{SDdef}, this time varying the
parameter $\gamma_{CIP}$ for three different
$\varepsilon$. Even when accounting for the increased number of
degrees of freedom in method \eqref{FEM} the error of
\eqref{standardFEM} is more than 50\% large in all the computations
and where \eqref{standardFEM} fails it is more than a factor 1000 larger.}
\begin{table}
\begin{center}
\begin{tabular}{|c|c|c|c|c|c|c|}
\hline
N & SG, $L^2$  & SG, SD & \eqref{standardFEM}, $L^2$ &
\eqref{standardFEM}, SD &  \eqref{FEM}, $L^2$ & \eqref{FEM}, SD\\
\hline
3 & 0.041 & 1.0   & 0.029 & 0.58 & 0.028 &  0.58 \\ \hline
4 & 0.025 & 0.88 &7.2E-3 &  0.20&  6.5E-3& 0.20  \\ \hline
5 & 0.010 & 0.48 & 1.7E-3 & 0.071 & 1.5E-3 & 0.069  \\ \hline
6 & 0.015 & 1.1 & 4.5E-4& 0.026 &4.0E-4  & 0.025  \\ \hline
7& 7.8E-3 & 0.76 & 1.1E-4& 9.1E-3 & 1.0E-4 & 8.7E-3  \\ \hline
8 & 1.9E-3 & 1.1 &2.5E-5 & 3.0E-3 & 2.4E-5  & 3.0E-3  \\ \hline
\end{tabular}
\caption{Errors of estimated quantities for the smooth
  solution approximated using piecewise affine elements.  SG means
  standard Galerkin and equations refer to methods used. ``L2''
  denotes the error in $L^2$-norm,  ``SD'' denotes the error in the
  streamline derivative norm defined in equation \eqref{SDdef}.}\label{SmoothP1}
\end{center}
\end{table} 
\begin{table}
\begin{center}
\begin{tabular}{|c|c|c|c|c|c|c|}
\hline
N & SG, $L^2$  & SG, SD & \eqref{standardFEM}, $L^2$ &
\eqref{standardFEM}, SD &  \eqref{FEM}, $L^2$ & \eqref{FEM}, SD\\ \hline
3 & 0.028 & 0.58 & 9.3E-4 & 0.060 & 7.5E-4 & 0.045 \\ \hline
4 & 4.6E-3 & 0.25 & 1.7E-4 & 0.014& 1.1E-4 & 8.7E-3\\ \hline
5 & 1.9E-3 & 0.17 & 2.7E-5 & 3.1E-3 & 1.4E-5 & 1.7E-3 \\ \hline
6 & 3.0E-4  & 0.042 &3.3E-6 & 5.1E-4& 1.7E-6 & 2.7E-4\\ \hline
7 & 3.3E-5 & 6.1E-3 & 4.4E-7 & 9.2E-5& 2.1E-7 & 4.7E-5\\ \hline
\end{tabular}
\caption{Errors of estimated quantities for the smooth
  solution approximated using piecewise quadratic elements.  SG means
  standard Galerkin and equations refer to methods used. ``L2'' denotes the error in $L^2$-norm,  ``SD'' denotes the error in the streamline derivative.}\label{SmoothP2}
\end{center}
\end{table} 
\begin{figure}
\centering
\includegraphics[width=6cm]{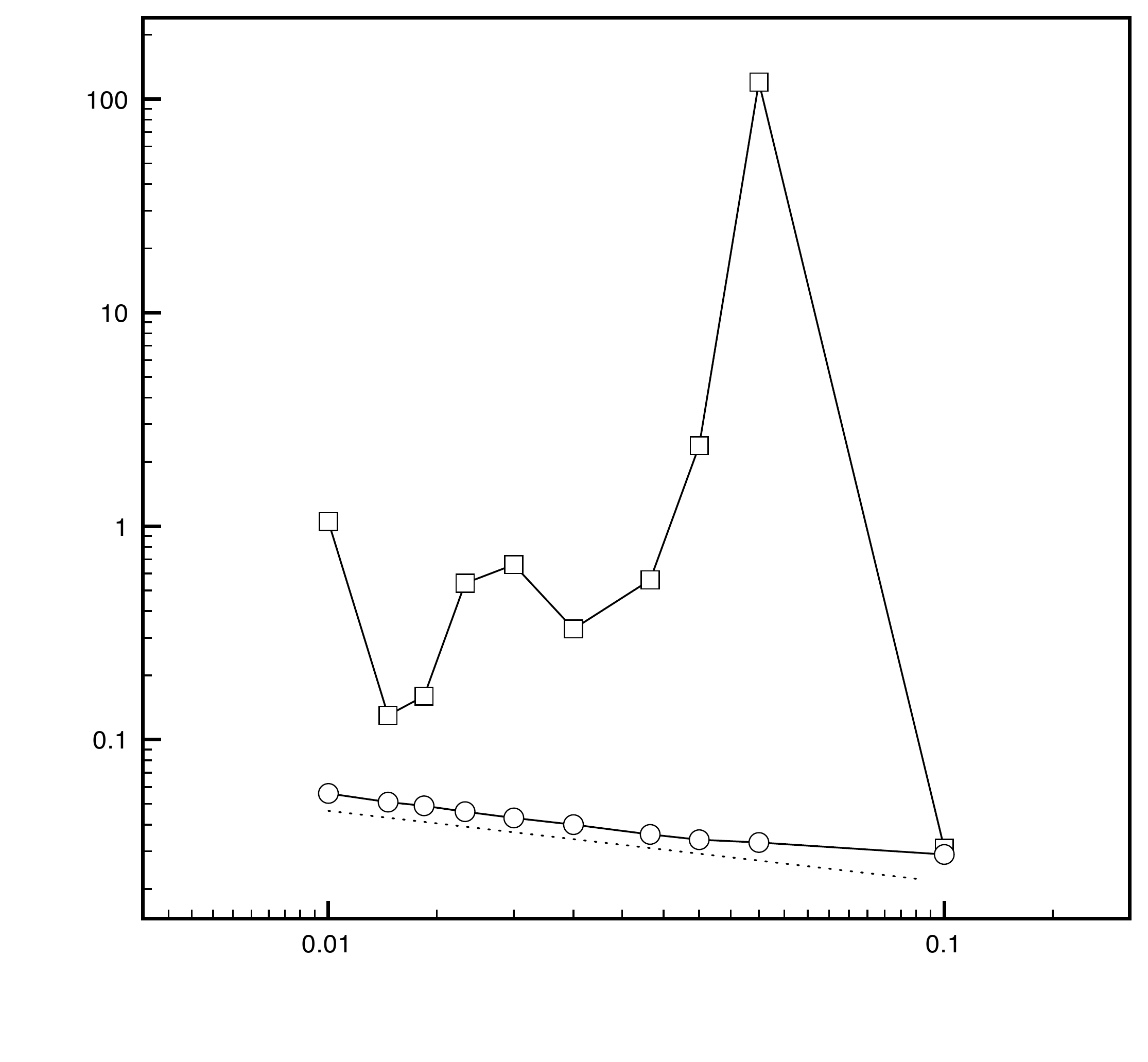}
\includegraphics[width=6cm]{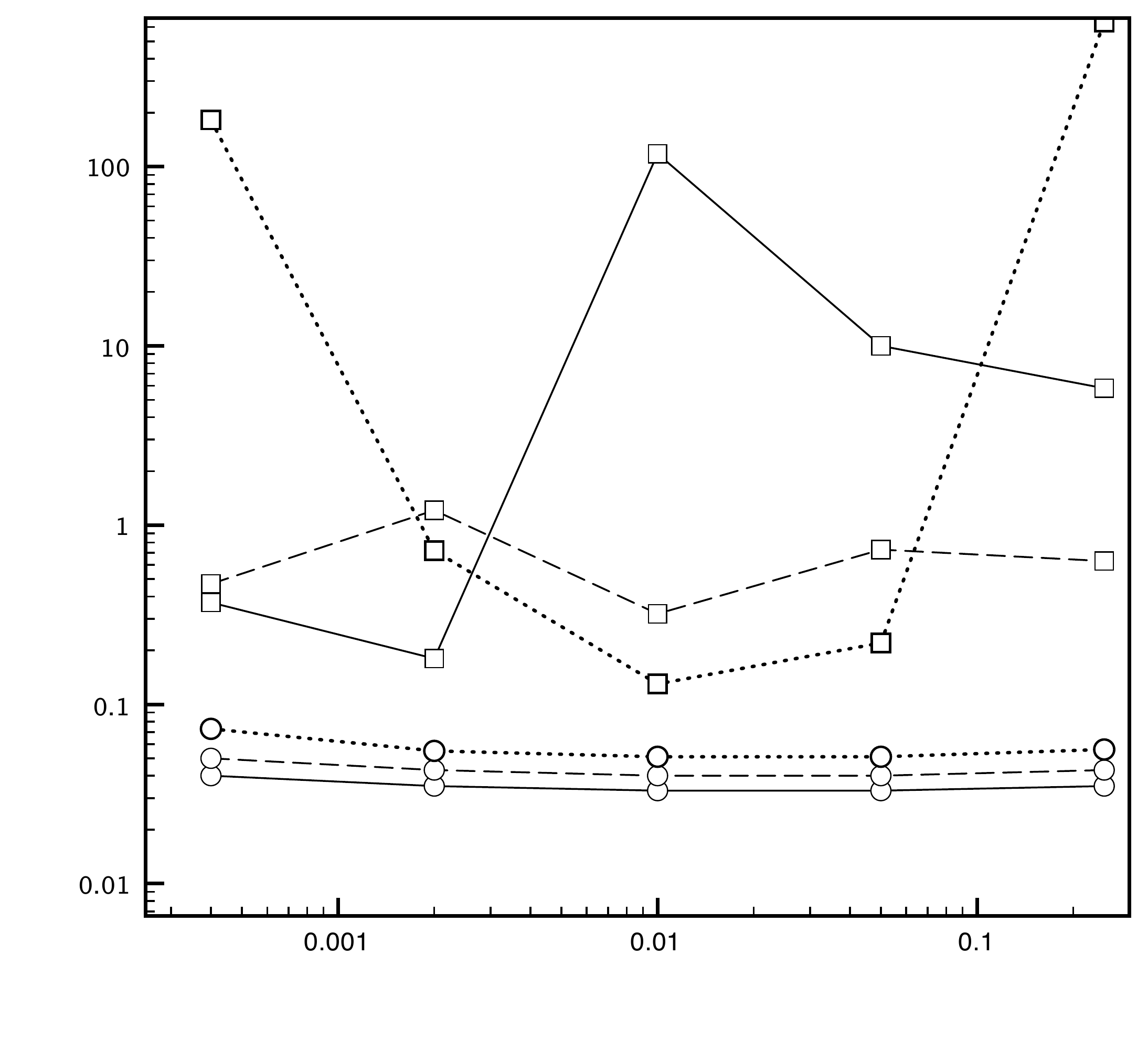}
\caption{Study of the error in the SD-norm 
error \eqref{SDdef}, circles: method \eqref{FEM}, squares: method \eqref{standardFEM}. Left: under variation of $\varepsilon$ in
\eqref{badbeta}, with $\gamma_{CIP}=0.01$, dotted line $O(\epsilon^{-\tfrac13})$. Right: under
variation of $\gamma_{CIP}$ for different $\epsilon$ (full line: $\epsilon =
0.05$; dashed line: $\epsilon=0.025$ and dotted line: $\epsilon=0.0125$)}\label{stab_study}
\end{figure}
\subsection{A data assimilation example}
Finally we consider a model problem for data assimilation where the
boundary conditions of the problem \eqref{comptransp} are imposed on
the outflow boundary instead of the inflow boundary. The method
\eqref{FEM} with the bilinear form \eqref{discrete_bilin} and the stabilizing term
\eqref{assim_stab}, with $X=CIP$ was applied. We consider the
test case with smooth solution \eqref{exact_sol} and the velocity field
\eqref{beta1}. In table \ref{dataassim} we give the computational
errors in the $L^2$-norm and the streamline norm \eqref{SDdef}, using either piecewise affine and piecewise
quadratic elements. Recalling the results in tables
\ref{SmoothP1} and \ref{SmoothP2} we see that the errors are
comparable. This is not surprising since the use of the adjoint
equation makes the two cases similar. Attempts to use \eqref{standardFEM} with weakly
imposed boundary conditions on the outflow 
and $\gamma_{CIP}>0$ were not fruitful. This is
expected since the stabilized methods on the form
\eqref{standardFEM} all are based on upwinding, which is unphysical in
this setting. 
{
Indeed the standard unstabilized Galerkin method
performs better than the standard stabilized method for this smooth
solution. When the stabilization parameter is chosen negative we
recover the expected behavior of the stabilized method. We give the
results of \eqref{standardFEM} using $\gamma_{bc}=-1.0$ and $\gamma_{CIP}=0.001$, $\gamma_{CIP}=0$, $\gamma_{CIP}=-0.01$ in table
\ref{standard_data_assim}.}
\begin{table}
\begin{center}
\begin{tabular}{|c|c|c|c|c|}
\hline
N & $\mathbb{P}_1$, $L^2$ &
$\mathbb{P}_1$, SD & $\mathbb{P}_2$, $L^2$ &
$\mathbb{P}_2$, SD \\
\hline
3 & 0.033 & 0.75 & 1.1E-3 &  0.052 \\ \hline
4 &7.1E-3 &  0.23&  1.5E-4 & 9.6E-3  \\ \hline
5 &1.6E-3 & 0.075 & 1.8E-5 & 1.8E-3  \\ \hline
6 & 4.1E-4& 0.026 & 2.0E-6  & 2.8E-4  \\ \hline
7& 1.0E-4& 8.9E-3 & 2.4E-7 & 4.8E-5  \\ \hline
8 & 2.4E-5 & 3.0E-3 & -- & -- \\ \hline
\end{tabular}
\caption{Data assimilation using \eqref{FEM}. Errors of estimated quantities for the smooth
  solution \eqref{exact_sol} computed with data given on the outflow
  boundary. Approximation using piecewise affine ($\mathbb{P}_1$) and
  quadratic ($\mathbb{P}_2$) elements.  ``L2'' denotes the error in $L^2$-norm,  ``SD'' denotes the error in the streamline derivative.}\label{dataassim}
\end{center}
\end{table} 
\begin{table}
\begin{center}
\begin{tabular}{|c|c|c|c|c|c|c|}
\hline
N & $\gamma_{1}$, $L^2$ &
$\gamma_{1}$, SD & $\gamma_{2}$, $L^2$ &
$\gamma_{2}$, SD & $\gamma_{3}$, $L^2$ &
$\gamma_{3}$, SD \\
\hline
3 & 0.044 & 3.48 & 0.034 &  2.8 &0.029 &  2.25\\ \hline
4 & 0.027 & 2.96& 0.01 & 1.2  & 6.7E-3&  0.74 \\ \hline
5 &0.27 & 31.0 & 2.7E-3 & 0.44  &1.6E-3 & 0.26  \\ \hline
6 & 2.74 & 455 & 1.1E-3  & 0.26 & 4.2E-4 & 0.094  \\ \hline
7& 6170 & 1.8E6 & 3.7E-4 & 0.11  &1.1E-4&  0.033 \\ \hline
8 & 67471 & 3.4E7 & 9.9E-5 & 0.041  & 2.5E-3& 0.011 \\ \hline
\end{tabular}
\caption{Data assimilation using the method \eqref{standardFEM} with
  the forms \eqref{discrete_bilin} and \eqref{assim_stab}, piecewise
  affine elements, $\gamma_{bc}=-1$
  and three different choices of $\gamma_{CIP}$ denoted by $\gamma_1$,
  $\gamma_2$ and $\gamma_3$. The CIP-stabilization parameters are
  assigned the values $\gamma_1=10^{-3}$, $\gamma_{2}=0$ and $\gamma_{3}=-10^{-2}$. Errors of estimated quantities for the smooth
  solution \eqref{exact_sol} computed with data given on the outflow
  boundary. ``L2'' denotes the error in $L^2$-norm,  ``SD'' denotes the error in the streamline derivative.}\label{standard_data_assim}
\end{center}
\end{table} 
\section{Concluding remarks}
We have extended the methods proposed in \cite{part1} to include
hyperbolic equations and shown how three stabilization methods known from the
litterature can be used to obtain stable and (quasi) optimally convergent approximations. Compared to
the standard stabilized method we show that the new method yields
existence of discrete solutions and (quasi) optimal error estimates under
much weaker assumptions on the mesh parameter (``$h^2$ small enough'' compared to
``$h^{\frac12}$ small enough''). We would like to stress that the method
proposed herein will not necessarily yield a more accurate solution than the
standard stabilized methods in cases where both methods work. The new
method however has increased robustness for noncoercive
problems. It also makes it easier to
incorporate other data than classical inflow boundary data. The idea of recasting the problem in an
optimisation framework opens interesting perspectives for optimal
control, inverse problems and data assimilation using observers.
\section*{Acknowledgment} Partial funding for this research was
provided by EPSRC
(Award number EP/J002313/1).
\bibliographystyle{plain}   

\end{document}